\documentclass[a4paper,11pt]{amsart}
\usepackage[utf8]{inputenc}

\usepackage[margin=1in]{geometry}
\usepackage{mathtools}
\usepackage{amsfonts}
\usepackage{mathrsfs}
\usepackage{amsthm}
\usepackage{amssymb}
\usepackage{bm}
\usepackage{tikz-cd}
\usepackage{tikz}
\usetikzlibrary{positioning}
\usepackage[colorlinks,linkcolor=blue,citecolor=blue,urlcolor=blue]{hyperref}
\usepackage[shortlabels]{enumitem}

\usepackage[textsize=tiny]{todonotes}
\usepackage[style=alphabetic,doi=false,isbn=false,url=false,maxnames=10,maxalphanames=10,sorting=nyt]{biblatex}
\DeclareDelimFormat{multicitedelim}{\addcomma\space}

\renewbibmacro*{journal+issuetitle}{%
  \usebibmacro{journal}%
  \iffieldundef{volume}
    { 
      \iffieldundef{year}
        {}
        {\setunit{\addspace}\printtext[parens]{\printfield{year}}}%
    }
    { 
      \setunit{\addspace}\printfield{volume}%
      \iffieldundef{year}
        {}
        {\setunit{\addspace}\printtext[parens]{\printfield{year}}}%
    }%
  \setunit*{\addcomma\space}%
  \printfield{number}%
  \newunit}

\DeclareFieldFormat[article, thesis, incollection, inbook]{title}{\mkbibemph{#1}}

\DeclareFieldFormat[article]{journaltitle}{#1}

\DeclareFieldFormat[article]{volume}{\textbf{#1}}

\DeclareFieldFormat[article]{number}{no.~#1}
 
\DeclareFieldFormat[incollection, inbook]{booktitle}{#1}

\renewbibmacro*{in:}{}

\DeclareFieldFormat{pages}{#1}

\newcommand{\C}{\mathbb{C}}
\newcommand{\N}{\mathbb{N}}
\newcommand{\Z}{\mathbb{Z}}

\newcommand{\Ac}{\mathcal{A}}
\newcommand{\Bc}{\mathcal{B}}
\newcommand{\Cc}{\mathcal{C}}
\newcommand{\Dc}{\mathcal{D}}
\newcommand{\Pc}{\mathcal{P}}
\newcommand{\Tc}{\mathcal{T}}
\newcommand{\Rc}{\mathcal{R}}
\newcommand{\Sc}{\mathcal{S}}

\newcommand{\Mf}{\mathfrak{M}^{\bullet}}
\newcommand{\Mft}{\widetilde{\mathfrak{M}}^{\bullet}}
\newcommand{\Tf}{\mathfrak{T}^{\bullet}}
\newcommand{\Tft}{\widetilde{\mathfrak{T}}^{\bullet}}

\newcommand{\fr}[1]{^{#1\operatorname{-fr}}}
\newcommand{\filt}[1]{^{#1\operatorname{-filt}}}

\newcommand{\tot}{\mathrm{tot}}
\newcommand{\sub}{\mathrm{sub}}
\newcommand{\quot}{\mathrm{quot}}
\newcommand{\subquot}{\mathrm{subquot}}
\newcommand{\Filt}{\mathrm{Filt}}
\newcommand{\CoFilt}{\mathrm{CoFilt}}
\newcommand{\op}{^\mathrm{op}}
\newcommand{\Hom}{\mathrm{Hom}}
\newcommand{\Ext}{\mathrm{Ext}}
\newcommand{\RHom}{\mathrm{RHom}}
\newcommand{\id}{\mathrm{id}}
\newcommand{\res}{\mathrm{res}}
\newcommand{\can}{\mathrm{can}}
\newcommand{\gpr}{\mathrm{gpr}}
\newcommand{\gin}{\mathrm{gin}}
\newcommand{\gprs}{\underline{\mathrm{gpr}}}
\newcommand{\gins}{\underline{\mathrm{gin}}}

\newcommand{\coker}{\operatorname{coker}}
\newcommand{\Mod}{\operatorname{Mod}}
\newcommand{\modcat}{\operatorname{mod}}

\newcommand{\ind}{\operatorname{ind}}

\newcommand{\cev}[1]{\reflectbox{\ensuremath{\vec{\reflectbox{\ensuremath{#1}}}}}}

\newtheorem{thm}{Theorem}[section]
\newtheorem{lemma}[thm]{Lemma}
\newtheorem{prop}[thm]{Proposition}
\newtheorem{cor}[thm]{Corollary}

\theoremstyle{definition}
\newtheorem{defn}[thm]{Definition}
\newtheorem{remark}[thm]{Remark}
\newtheorem{notation}[thm]{Notation}
\newtheorem{example}[thm]{Example}
\newtheorem{assumption}[thm]{Assumption}

\title[Categories of split filtrations and graded quiver varieties]{Categories of split filtrations and graded quiver varieties}
\author{Ricardo Canesin}
\address{Université Paris Cité, Sorbonne Université, CNRS, IMJ-PRG, F-75013 Paris, France}
\email{ricardo.canesin@imj-prg.fr}


\begin{document}

\begin{abstract}
    By the work of Hernandez--Leclerc, Leclerc--Plamondon, and Keller--Scherotzke, affine graded Nakajima quiver varieties associated with a Dynkin quiver $Q$ admit an algebraic description in terms of modules over the singular Nakajima category $\mathcal{S}$ and a stratification functor to the derived category of $Q$. In this paper, we extend this framework to Nakajima's $n$-fold affine graded tensor product varieties, which allow one to geometrically realize $n$-fold tensor products of standard modules over the quantum affine algebra. We introduce a category of filtrations with splitting of length~$n$ of modules over a category and show that it is equivalent to the module category of a triangular matrix category. Applied to the singular Nakajima category, this yields a category $\mathcal{S}^{n\operatorname{-filt}}$ whose modules are parametrized by the points of the $n$-fold tensor product varieties. Generalizing the results of Keller--Scherotzke from $\mathcal{S}$ to $\mathcal{S}^{n\operatorname{-filt}}$, we prove that the stable category of pseudo-coherent Gorenstein projective $\mathcal{S}^{n\operatorname{-filt}}$-modules is triangle equivalent to the derived category of the algebra of $n \times n$ upper triangular matrices over the path algebra of $Q$, and we obtain a corresponding stratification functor.
\end{abstract}

\maketitle

\setcounter{tocdepth}{1}
\tableofcontents

\section{Introduction}

The theory of (graded) quiver varieties and its applications to representation theory were developed by Nakajima in a series of papers \cite{Nakajima94,Nakajima98,Nakajima01findim,Nakajima04}, motivated notably by earlier work of Ringel \cite{Ringel90} and Lusztig \cite{Lusztig90}. In particular, Nakajima used these varieties to give a geometric realization of certain finite-dimensional representations of the quantum affine algebra $U_q(\widehat{\mathfrak{g}})$ associated with a finite-dimensional complex simple Lie algebra $\mathfrak{g}$, thereby obtaining important information about their structure and $q$-character theory. The graded version of quiver varieties also plays a prominent role in Nakajima's study of cluster algebras and their monoidal categorifications \cite{Nakajima11}, which was subsequently generalized by Kimura--Qin \cite{KimuraQin14} to graded quiver varieties associated with arbitrary acyclic quivers.

A fruitful approach to the study of graded quiver varieties was pioneered by Hernandez--Leclerc \cite{HernandezLeclerc15} and further developed by Leclerc--Plamondon \cite{LeclercPlamondon13} and Keller--Scherotzke \cite{KellerScherotzke16}. In the latter two works, affine graded quiver varieties were shown to be isomorphic to representation varieties of a category $\Sc$, called the \emph{singular Nakajima category}. Moreover, for any Dynkin quiver $Q$ of the same type as $\mathfrak{g}$, Keller and Scherotzke constructed a functor
\begin{equation}\label{eq:stratification functor of KS in introduction}
\Phi: \modcat\Sc \longrightarrow \Dc^b(\modcat kQ)
\end{equation}
from the category of finite-dimensional $\Sc$-modules to the bounded derived category of representations of $Q$ over the field $k = \C$ of complex numbers. They proved that this functor recovers Nakajima's regular stratification on the affine graded quiver variety in the following sense: two $\Sc$-modules $M$ and $N$ lie in the same stratum if and only if $\Phi(M)$ and $\Phi(N)$ are isomorphic in $\Dc^b(\modcat kQ)$. These results have found several applications, particularly in the theory of quantum groups and their representations (see, for instance, \cite{KellerScherotzke14, Qin16, ScherotzkeSibilla16, LuWang21, Fujita22}).

Our goal in this paper is to extend this framework to another class of varieties defined independently by Nakajima \cite{Nakajima01} (see also \cite{Nakajima11} for the graded version), Malkin \cite{Malkin03}, and Varagnolo--Vasserot \cite{VaragnoloVasserot03}. We call these varieties the \emph{$n$-fold affine graded tensor product varieties}, where $n \geq 1$ is a positive integer. Through Nakajima's constructions, they provide a geometric realization of the tensor product of a sequence of $n$ standard modules over the quantum affine algebra. These varieties have applications, for example, as valuable tools in the geometric realization of the quantum affine Schur--Weyl duality \cite{Fujita20} and in the study of $R$-matrices and their denominators \cite{Fujita22, FujitaHernandez25}.

Our first step is to identify the $n$-fold affine graded tensor product variety with an appropriate representation variety. This is achieved by considering a category of filtrations of $\Sc$-modules of length $n$ equipped with a splitting, which we call \emph{split filtrations} for short. Let us describe this construction more precisely, and in greater generality, as it may be of independent interest. Fix a base field $k$. Let $\Ac$ be a small $k$-category and let $\Bc$ be a (not necessarily full) subcategory of $\Ac$ containing all objects. The \emph{category of split filtrations of $\Ac$-modules over $\Bc$} is the category $\Filt^n_{\Bc}(\Ac)$ whose objects are $\Ac$-modules $M$ equipped with a filtration
\[
0 = M_0 \subseteq M_1 \subseteq \dotsb \subseteq M_{n-1} \subseteq M_n = M
\]
by $\Ac$-submodules together with $\Bc$-linear maps $r_i^M: M_i \to M_{i-1}$ that are retractions for the inclusions $M_{i-1} \to M_i$. In particular, when restricting the module structure to $\Bc$, each inclusion $M_{i-1} \to M_i$ becomes a split monomorphism. We prove that $\Filt^n_{\Bc}(\Ac)$ is itself a module category.

\begin{thm}[Proposition \ref{prop:characterization of Filt as a module category}]\label{thm:split filtrations form module category (introduction)}
    The category $\Filt_{\Bc}^n(\Ac)$ is equivalent to $\Mod\Tc_{\Bc}^n(\Ac)$ for a certain $k$-category $\Tc_{\Bc}^n(\Ac)$.
\end{thm}

The category $\Tc_{\Bc}^n(\Ac)$ is an example of a \emph{triangular matrix category}. This notion generalizes triangular matrix rings, in the spirit of Mitchell's philosophy that ``additive categories are rings with several objects'' \cite{Mitchell72}. Triangular matrix categories have appeared in various forms in the literature (see, for example, \cite{Tabuada08, KuznetsovLunts15, LeonOrtizSantiago23}). For instance, when $n = 2$, one has
\[
\Tc_{\Bc}^2(\Ac) = \begin{pmatrix}
    \Ac & 0\\
    \ker(\Ac \otimes_{\Bc} \Ac \to \Ac) & \Ac
\end{pmatrix},
\]
where the $\Ac$-$\Ac$-bimodule on the lower-left corner is the kernel of the morphism induced by the composition of $\Ac$. A similar description holds for $\Tc^n_{\Bc}(\Ac)$ for arbitrary $n$ (see Lemma \ref{lemma:explicit description of the triangular category}).

We are primarily interested in the case where $\Ac = \Sc$ is the singular Nakajima category and $\Bc = k\Sc_0$ is the smallest $k$-subcategory of $\Ac$ containing all objects. We define the \emph{singular Nakajima category of split- filtrations of length $n$} as the category $\Sc\filt{n} = \Tc^n_{k\Sc_0}(\Sc)$. As an immediate consequence of \cite[Theorem 2.4]{LeclercPlamondon13} and the proof of \cite[Lemma 3.6]{Nakajima01}, the $n$-fold affine graded tensor product varieties are realized as representation varieties of $\Sc\filt{n}$. Although these varieties only depend on the filtration by $\Sc$-submodules, we emphasize that the splitting over $k\Sc_0$ is encoded in the group action considered in \cite{Nakajima01} and is necessary for Theorem \ref{thm:split filtrations form module category (introduction)} to hold.

We also provide an alternative description of $\Sc\filt{n}$ in terms of mesh categories, building on Keller--Scherotzke's construction of the category $\Sc$ (see Proposition \ref{prop:modules over Sfilt are the same as Filt(S)}). As a consequence, we show in Section \ref{section:generalizing KS} that many results and arguments in \cite{KellerScherotzke16} extend to our setting. The realization of $\Sc\filt{n}$ as a triangular matrix category plays a crucial role in these generalizations, particularly in the proof of Lemma \ref{lemma:crucial lemma to prove weakly Gorenstein}. Our main result is a generalization of \cite[Theorem 5.18]{KellerScherotzke16}, describing the stable category $\gprs(\Sc\filt{n})$ of (pseudo-coherent) Gorenstein projective $\Sc\filt{n}$-modules, defined via totally acyclic complete resolutions whose terms are finitely generated projective.

\begin{thm}[Theorem \ref{thm:equivalences for stable categories}]
    There is an equivalence of triangulated categories
    \[
    \gprs(\Sc\filt{n}) \xlongrightarrow{\sim} \Dc^b(\modcat k\cev{\mathsf{A}}_n \otimes kQ),
    \]
    where $k\cev{\mathsf{A}}_n$ denotes the path algebra of a linearly oriented Dynkin quiver of type $\mathsf{A}_n$.
\end{thm}

\begin{remark}
    When $n=1$, this equivalence was proved by Keller--Scherotzke. In ongoing work, we are developing an alternative proof for the theorem above using this base case. The realization of $\Sc\filt{n}$ as a triangular matrix category allows us to express a suitable singularity category of $\Sc\filt{n}$ as a gluing, in the sense of \cite[Section 4]{KuznetsovLunts15}. We expect to recover the equivalence above from this decomposition.
\end{remark}

We define the \emph{stratification functor} $\Phi^n$ from the category of finite-dimensional $\Sc\filt{n}$-modules to $\Dc^b(k\cev{\mathsf{A}}_n \otimes kQ)$ as the composition
\[
\modcat\Sc\filt{n} \xlongrightarrow{\Omega} \gprs(\Sc\filt{n}) \xlongrightarrow{\sim} \Dc^b(k\cev{\mathsf{A}}_n \otimes kQ),
\]
where $\Omega$ is the syzygy functor, sending an $\Sc\filt{n}$-module $M$ to the kernel of a projective cover $P \to M$. When $n = 1$, this recovers Keller--Scherotzke's functor $\Phi$ in (\ref{eq:stratification functor of KS in introduction}). We show that $\Phi^n$ is closely related to the stratification of the affine graded quiver variety.

\begin{thm}[Corollary \ref{cor:stratum of subquotients}]\label{thm:strata of subquotients (introduction)}
    Let $M$ and $N$ be two points in an $n$-fold affine graded tensor product variety, viewed as $\Sc\filt{n}$-modules. If $\Phi^n(M) \cong \Phi^n(N)$, then for any $0 \leq i < j \leq n$, the subquotients $M_j/M_i$ and $N_j/N_i$ of the filtrations of $M$ and $N$ lie in the same stratum of an affine graded quiver variety.
\end{thm}

As shown in Example \ref{example:infinite partition}, the converse does not hold in general. Nevertheless, it would be interesting to determine when it does, and more generally, to give a geometric description of the partition on the $n$-fold affine graded tensor product variety induced by $\Phi^n$.

\subsection*{Organization} This paper is organized as follows. In Section \ref{section:preliminaries}, we review the background material needed for our results. We recall the definition and basic properties of modules over categories, introduce triangular matrix categories and describe their module categories, and review the construction of mesh categories and Happel's embedding. In Section \ref{section:split filtrations}, we define and study the category of split filtrations, proving Theorem \ref{thm:split filtrations form module category (introduction)} and establishing useful properties of the canonical functors introduced in Section \ref{section:canonical functors}. In Section \ref{section:generalizing KS}, we study the category of split filtrations of the singular Nakajima category and explain its connection to graded tensor product varieties. We generalize many results and arguments of \cite{KellerScherotzke16} to our setting. For example, we explicitly describe the projective and injective (co)resolutions of simple $\Sc\filt{n}$-modules, prove that $\Sc\filt{n}$ is weakly Gorenstein, and characterize its stable category of Gorenstein projective (and injective) modules. Finally, we establish compatibility formulas between the stratification functors for different values of $n$, leading to Theorem \ref{thm:strata of subquotients (introduction)}.

\subsection*{Acknowledgements} I thank Geoffrey Janssens for suggesting this project and for many helpful discussions that shaped the ideas in this work. I also thank my advisor, Bernhard Keller, for his constant guidance and support, both in developing the mathematics and in improving the exposition of this paper.

\section{Preliminaries}\label{section:preliminaries}

Throughout this paper, we fix a field $k$. All our categories are $k$-categories, that is, enriched over the category $\Mod k$ of $k$-vector spaces. Unadorned tensor products are taken over $k$.

\subsection{Modules over categories} Let $\Ac$ be a small $k$-category. A \emph{right $\Ac$-module} is a $k$-linear functor $\Ac\op \to \Mod k$. We denote by $\Mod\Ac$ the category of right $\Ac$-modules, which is an abelian category. Similarly, a \emph{left $\Ac$-module} is a $k$-linear functor $\Ac \to \Mod k$. If not stated, a module is by default a \emph{right} module.

\begin{notation}
Let $M$ be a right $\Ac$-module and $a: x \to y$ a morphism in $\Ac$. If $m \in M(y)$, we will denote $M(a)(m) \in M(x)$ simply by $m \cdot a$ or $ma$. A similar notation applies to left modules.
\end{notation}

Let $x \in \Ac$. The \emph{free module} (or the right representable module) is the functor $x^{\wedge} = \Ac(-,x): \Ac\op \to \Mod k$. The \emph{cofree module} is the functor $x^{\vee} = D(\Ac(x,?)): \Ac\op \to \Mod k$, where $D$ denotes the duality over $k$. We have canonical isomorphisms
\[
\Hom(x^{\wedge}, M) = M(x) \quad \textrm{and} \quad \Hom(M,x^{\vee}) = D(M(x))
\]
for any $\Ac$-module $M$, so that free (resp. cofree) modules are projective (resp. injective). We say that a module $M$ is \emph{finitely generated} (resp. \emph{finitely cogenerated}) if it is isomorphic to a quotient (resp. a submodule) of a finite direct sum of free (resp. cofree) modules.

We will be mainly interested in the case where $\Ac$ satisfies the following conditions.

\begin{assumption}\label{assumption:Hom-finite and directed}
    The category $\Ac$ is
    \begin{enumerate}[(1)]
        \item \emph{$\Hom$-finite}, that is, its morphism spaces are finite-dimensional;
        \item \emph{directed}, that is, the endomorphism algebra of each object is $k$, and there is a partial order on the set of objects such that $\Ac(x,y) \neq 0$ implies $x \leq y$.
    \end{enumerate}
    We additionally assume that the partial order above is \emph{locally finite}, that is, every bounded interval is finite.
\end{assumption}

For $x \in \Ac$, let $S_x$ be the $\Ac$-module with $S_x(x) = k$ and $S_x(y) = 0$ for $y \neq x$, and the evident action on morphisms. By the directedness assumption, every finite-dimensional simple $\Ac$-module is of this form, and $S_x$ is the unique simple quotient of $x^{\wedge}$ and the unique simple submodule of $x^{\vee}$.

An $\Ac$-module $M$ is \emph{pointwise finite-dimensional} if $M(x)$ is finite-dimensional for all $x \in \Ac$. It is \emph{right bounded} (resp. \emph{left bounded}) if there are objects $x_1,\dots,x_n \in \Ac$ such that $M(x) \neq 0$ implies $x \leq x_i$ (resp. $x \geq x_i$) for some $1 \leq i \leq n$. For example, finitely generated projective modules are right bounded, while finitely cogenerated injective modules are left bounded. Both are pointwise finite-dimensional by the $\Hom$-finiteness assumption.

If $M$ is pointwise finite-dimensional and right bounded, then $M$ admits a projective cover $P_0 \to M$ where $P_0$ is a (possibly infinite) coproduct of free modules. The multiplicity of $x^{\wedge}$ in $P_0$ is the dimension of $\Hom(M,S_x)$, which is finite. The kernel of $P_0 \to M$ is again right bounded and pointwise finite-dimensional (by the locally finiteness hypothesis), so that we can iterate this procedure. Consequently, any such module $M$ admits a minimal projective resolution
\[
\dotsb \longrightarrow P_2 \longrightarrow P_1 \longrightarrow P_0 \longrightarrow M \longrightarrow 0,
\]
where $P_i$ is a coproduct of free modules and the multiplicity of $x^{\wedge}$ in $P_i$ is given by the dimension of $\Ext^i(M,S_x)$. We have a dual statement for computing injective coresolutions of pointwise finite-dimensional left bounded modules.

\subsection{Bimodules and tensor products} Let $\Ac$ and $\Bc$ be (small) $k$-categories. An $\Ac$-$\Bc$-bimodule is a $k$-linear functor $\Bc\op \otimes \Ac \to \Mod k$. Here, the tensor product $\Bc\op \otimes \Ac$ is the category whose objects are pairs of objects in $\Bc\op$ and $\Ac$, and whose morphism spaces are the tensor products of those in $\Bc\op$ and in $\Ac$. If $X: \Bc\op \otimes \Ac \to \Mod k$ is an $\Ac$-$\Bc$-bimodule, we may specialize at objects $a \in \Ac$ and $b \in \Bc$ to obtain a right $\Bc$-module $X(-,a)$ or a left $\Ac$-module $X(b,?)$. For example, the $\Hom$-functor $\Ac(-,?): \Ac\op \otimes \Ac \to \Mod k$ is an $\Ac$-bimodule whose specializations yield the left and right representable modules. We will also call it the \emph{regular $\Ac$-bimodule} and denote it by $\Ac$.

Let $\Ac$, $\Bc$ and $\Cc$ be $k$-categories. If $X$ is an $\Ac$-$\Bc$-bimodule and $Y$ is a $\Bc$-$\Cc$-bimodule, the tensor product $X \otimes_{\Bc} Y$ is the $\Ac$-$\Cc$-bimodule defined on objects $a \in \Ac$ and $c \in \Cc$ as
\[
(X \otimes_{\Bc} Y)(c,a) = \coker\left(\bigoplus_{b_1,b_2 \in \Bc}X(b_2,a) \otimes \Bc(b_1,b_2) \otimes Y(c,b_1) \xrightarrow{\nu} \bigoplus_{b \in \Bc} X(b,a) \otimes Y(c,b)\right),
\]
where $\nu(x \otimes f \otimes y) = xf \otimes y - x \otimes fy$. Since $X \otimes_{\Bc} Y$ is again a bimodule, this construction may be iterated. It behaves as the tensor product of bimodules over ordinary rings.

\subsection{Triangular matrix categories} Let $\Ac$ and $\Bc$ be small $k$-categories. Given a $\Bc$-$\Ac$ bimodule $X: \Ac\op \otimes \Bc \to \Mod k$, we define the \emph{(lower) triangular matrix category}
\[
\Tc = \Tc(\Ac,X,\Bc) = \begin{pmatrix}
    \Ac & 0\\
    X & \Bc
\end{pmatrix}
\]
to be the $k$-category whose set of objects is the disjoint union of the set of objects of $\Ac$ and $\Bc$, and whose morphism spaces are given as follows:
\[
\Tc(x,y) = \begin{cases}
    \Ac(x,y), &\textrm{if } x,y \in \Ac,\\
    \Bc(x,y), &\textrm{if } x,y \in \Bc,\\
    X(x,y), &\textrm{if } x \in \Ac, y \in \Bc,\\
    0, &\textrm{if } x \in \Bc, y \in \Ac.
\end{cases}
\]
The composition law of $\Tc$ is induced by those of $\Ac$ and $\Bc$, and by the bimodule structure of $X$.

\begin{remark}
    If $\Ac$ and $\Bc$ contain a single object each, then we may view $\Ac$ and $\Bc$ as $k$-algebras. In this case, we can form the triangular matrix algebra associated with $\Ac$, $\Bc$, and $X$. It is isomorphic to 
    the matrix ring of the above category $\Tc$ in the sense of \cite[Section 7]{Mitchell72} and, consequently, it is Morita equivalent to $\Tc$.
\end{remark}

\begin{remark}
    Our definition is taken from \cite{Tabuada08}, where Tabuada defines the (upper) triangular dg category associated with two small dg categories and a dg bimodule. In this setting, a related construction can be found in \cite{KuznetsovLunts15}. If $\Ac$ and $\Bc$ are additive categories, the notion of a triangular matrix category also appears in \cite{LeonOrtizSantiago23}. With the appropriate identifications, their category is equivalent to the ``additivization'' of the category $\Tc$ above, that is, the category of finitely generated projective $\Tc$-modules.
\end{remark}

As proved in \cite{AuslanderReitenSmalo95} (and in \cite{LeonOrtizSantiago23} for the version with categories), the category of $\Tc$-modules is equivalent to the comma category (see \cite[Section II.6]{MacLane98}) associated with the pair of functors given by $-\otimes_{\Bc}X: \Mod\Bc \to \Mod\Ac$ and the identity functor on $\Mod\Ac$. Its objects are triples $(M_1,M_2,\varphi)$ where $M_1$ is an $\Ac$-module, $M_2$ is a $\Bc$-module, and $\varphi: M_2 \otimes_{\Bc} X \to M_1$ is a morphism of $\Ac$-modules. A morphism $(M_1,M_2,\varphi) \to (N_1,N_2,\psi)$ is given by morphisms $f_i: M_i \to N_i$ for $i=1,2$ such that the following square commutes:
\[
\begin{tikzcd}
	{M_2 \otimes_{\Bc} X} & {N_2 \otimes_{\Bc} X} \\
	{M_1} & {N_1}
	\arrow["{f_2 \otimes \id}", from=1-1, to=1-2]
	\arrow["{\varphi}"', from=1-1, to=2-1]
	\arrow["{\psi}", from=1-2, to=2-2]
	\arrow["{f_1}"', from=2-1, to=2-2]
\end{tikzcd}\]
The equivalence takes a triple $(M_1,N_1,\varphi)$ to the $\Tc$-module which is best thought as the ``row module'' $\begin{pmatrix}
    M_1 & M_2
\end{pmatrix}$, where $\Tc$ acts by ``matrix multiplication'' on the right. From now on, we will identify $\Mod\Tc$ with this comma category. The category $\Mod\Tc\op$ of left $\Tc$-modules can also be identified with a comma category whose objects are triples $(M'_1,M'_2,\varphi')$ where $M'_1$ is a left $\Ac$-module, $M'_2$ is a left $\Bc$-module, and $\varphi': X \otimes_{\Ac} M'_1 \to M_2'$ is a morphism of left $\Bc$-modules. 

\begin{remark}\label{rem:triangular matrix categories of arbitrary order}
By iterating the definitions above, we may construct triangular matrix categories of arbitrary order. Following \cite{Wang16}, they can be explicitly described as follows. Let $\Ac_1, \Ac_2, \dots, \Ac_n$ be small $k$-categories and, for all $n \geq i > j \geq 1$, let $X_{ij}$ be an $\Ac_i$-$\Ac_j$-bimodule. Suppose we are given bimodule morphisms $\mu_{ikj}: X_{ik} \otimes_{\Ac_k} X_{kj} \to X_{ij}$ such that the square
\begin{equation}\label{eq:associativity for matrix rings}
    \begin{tikzcd}[column sep=6em]
	{X_{ik} \otimes_{\Ac_k} X_{kl} \otimes_{\Ac_l} X_{lj}} & {X_{ik} \otimes_{\Ac_k} X_{kj}} \\
	{X_{il} \otimes_{\Ac_l} X_{lj}} & {X_{ij}}
	\arrow["{\id \otimes \mu_{klj}}", from=1-1, to=1-2]
	\arrow["{\mu_{ikl} \otimes \id}"', from=1-1, to=2-1]
	\arrow["{\mu_{ikj}}", from=1-2, to=2-2]
	\arrow["{\mu_{ilj}}"', from=2-1, to=2-2]
    \end{tikzcd}
\end{equation}
commutes for all $n \geq i > k > l > j \geq 1$. This data defines the triangular matrix category
\[
\Tc = \begin{pmatrix}
    \Ac_1 & 0 & 0 & \dotsb & 0\\
    X_{21} & \Ac_2 & 0 & \dotsb & 0\\
    X_{31} & X_{32} & \Ac_3 & \dotsb & 0\\
    \vdots & \vdots & \vdots & \ddots & \vdots\\
    X_{n1} & X_{n2} & X_{n3} & \dotsb & \Ac_n
\end{pmatrix},
\]
whose set of objects is the disjoint union of the set of objects of the $\Ac_i$ for $1 \leq i \leq n$. If $x \in \Ac_j$ and $y \in \Ac_i$, then
\[
\Tc(x,y) = \begin{cases}
    X_{ij}(x,y), &\textrm{if }i > j,\\
    \Ac_i(x,y), &\textrm{if }i = j,\\
    0, &\textrm{if }i < j.
\end{cases}
\]
The composition law is induced by the bimodule structures and by the maps $\mu_{ikj}$. The square (\ref{eq:associativity for matrix rings}) ensures that the composition is associative.

A $\Tc$-module is equivalent to the following data: a tuple $(M_1,\dots,M_n, \varphi_{ij})$ where $M_i$ is an $\Ac_i$-module and $\varphi_{ij}: M_i \otimes_{\Ac_i} X_{ij} \to M_j$ (for $i > j$) is a morphism of $\Ac_j$-modules such that the square
\[\begin{tikzcd}[column sep=6em]
	{M_i \otimes_{\Ac_i} X_{ij} \otimes_{\Ac_j} X_{jk}} & {M_i \otimes_{\Ac_i} X_{ik}} \\
	{M_j \otimes_{\Ac_j} X_{jk}} & {M_k}
	\arrow["{\id \otimes \mu_{ijk}}", from=1-1, to=1-2]
	\arrow["{\varphi_{ij} \otimes \id}"', from=1-1, to=2-1]
	\arrow["{\varphi_{ik}}", from=1-2, to=2-2]
	\arrow["{\varphi_{jk}}"', from=2-1, to=2-2]
\end{tikzcd}\]
commutes for all $n \geq i > j > k \geq 1$. In this setting, a morphism $f: (M_1,\dots,M_n,\varphi_{ij}) \to (N_1,\dots,N_n,\psi_{ij})$ corresponds to morphisms $f_i: M_i \to N_i$ for all $1 \leq i \leq n$ such that the square
\begin{equation}\label{eq:morphisms in the comma category}
\begin{tikzcd}
	{M_i \otimes_{\Ac_i} X_{ij}} & {N_i \otimes_{\Ac_i} X_{ij}} \\
	{M_j} & {N_j}
	\arrow["{f_i \otimes \id}", from=1-1, to=1-2]
	\arrow["{\varphi_{ij}}"', from=1-1, to=2-1]
	\arrow["{\psi_{ij}}", from=1-2, to=2-2]
	\arrow["{f_j}"', from=2-1, to=2-2]
\end{tikzcd}
\end{equation}
commutes for all $i > j$.
\end{remark}

\subsection{The eight functors}\label{section:eight functors} Let $\Tc = \Tc(\Ac,X,\Bc)$ be a triangular matrix category as in the previous section. In \cite{Ladkani11}, it is shown that $\Mod\Tc$ can be viewed as a gluing of $\Mod\Ac$ and $\Mod\Bc$. This is described by a diagram of functors
\[\begin{tikzcd}
	{\Mod\Bc} & {\Mod\Tc} & {\Mod\Ac}
	\arrow["i_*", shift left = 2.5, from=1-1, to=1-2]
	\arrow["i_!", bend right=50, from=1-1, to=1-2]
	\arrow["i^{-1}"', shift left = 2.5, from=1-2, to=1-1]
	\arrow["i^!"', bend right=50, from=1-2, to=1-1]
	\arrow["j^{-1}", shift left = 2.5, from=1-2, to=1-3]
	\arrow["j^{\natural}", bend right=50, from=1-2, to=1-3]
	\arrow["j_!"', shift left = 2.5, from=1-3, to=1-2]
	\arrow["j_*"', bend right=50, from=1-3, to=1-2]
\end{tikzcd}\]
where we have adjunctions $i_! \dashv i^{-1} \dashv i_* \dashv i^!$ and $j^{\natural} \dashv j_! \dashv j^{-1} \dashv j_*$. These functors are defined as follows. For modules $A \in \Mod\Ac$, $B \in \Mod\Bc$ and $(M_1,M_2,\varphi) \in \Mod\Tc$, we have
\[
i_*(B) = (0,B,0), \quad j_!(A) = (A,0,0) \quad \textrm{and} \quad i_!(B) = (B \otimes_{\Bc} X, B, \id),
\]
and
\[
j^{-1}(M_1,M_2,\varphi) = M_1, \quad i^{-1}(M_1,M_2,\varphi) = M_2 \quad \textrm{and} \quad j^{\natural}(M_1,M_2,\varphi) = \coker\varphi.
\]
The action on morphisms is the evident one. For the functors $i^!$ and $j_*$, it is easier to use an alternative description of the modules over $\Tc$. By the tensor-$\Hom$ adjunction, we may equivalently view $\Tc$-modules as triples $(M_1,M_2,\psi)$ where $M_1$ is an $\Ac$-module, $M_2$ is a $\Bc$-module and $\psi: M_2 \to \Hom_{\Ac}(X,M_1)$ is a morphism of $\Bc$-modules. With this description, we have
\[
i^!(M_1,M_2,\psi) = \ker\psi \quad \textrm{and} \quad j_*(A) = (A, \Hom_{\Ac}(X,A), \id)
\]
for $A \in \Mod\Ac$ and $(M_1,M_2,\psi) \in \Mod\Tc$.

\subsection{Mesh categories and Happel's embedding} Let $Q = (Q_0,Q_1,s,t)$ be a \emph{quiver}, that is, a directed graph with set of vertices $Q_0$ and set of arrows $Q_1$, together with maps $s,t: Q_1 \to Q_0$ indicating the source and target of each arrow, respectively. We always suppose that $Q$ comes with a distinguished subset $F \subseteq Q_0$ (possibly empty) of \emph{frozen vertices}, so that $(Q,F)$ is an \emph{ice quiver}. For simplicity, we omit $F$ from the notation and say directly that $Q$ is an ice quiver. We also assume that $Q$ is finite and acyclic, that is, $Q_0$ and $Q_1$ are finite, and there are no oriented cycles.

The \emph{repetition quiver} $\Z Q$ associated with $Q$ is the ice quiver with set of vertices $(\Z Q)_0 = Q_0 \times \Z$ and with the following set of arrows: given $\alpha: i \to j$ in $Q_1$ and $p \in \Z$, there are arrows $(\alpha,p): (i,p) \to (j,p)$ and $\sigma(\alpha,p): (j,p-1) \to (i,p)$ in $(\Z Q)_1$. A vertex $(i,p)$ is frozen in $\Z Q$ if $i$ is frozen in $Q$. We define the bijection $\sigma: (\Z Q)_1 \to (\Z Q)_1$ by $\sigma(\beta) = \sigma(\alpha,p)$, if $\beta = (\alpha,p)$, and $\sigma(\beta) = (\alpha,p-1)$, if $\beta = \sigma(\alpha,p)$. We have an automorphism $\tau: \Z Q \to \Z Q$ which sends a vertex $(i,p)$ to $(i,p-1)$ and an arrow $\beta$ to $\sigma^2(\beta)$.

The \emph{path category} $\Pc_k(\Z Q)$ is the $k$-category whose objects are the vertices of $\Z Q$ and whose space of morphisms from a vertex $x$ to a vertex $y$ is the $k$-vector space with basis given by the paths from $x$ to $y$ in $\Z Q$. Composition is induced by the concatenation of paths, and the identity morphisms are the stationary paths of length zero. We now define the \emph{mesh category} $k(\Z Q)$ (after \cite{Riedtmann80,Gabriel80,KellerScherotzke16}) as the quotient of $\Pc_k(\Z Q)$ by the $k$-ideal generated by the morphisms
\begin{equation}\label{eq:mesh relation}
r_x = \sum_{\substack{\beta \in (\Z Q)_1\\ \beta: y \to x}}\beta \circ \sigma(\beta): \tau(x) \longrightarrow x
\end{equation}
for all \emph{non-frozen} vertices $x \in (\Z Q)_0$. Since we assume $Q$ is finite and acyclic, notice that $\Pc_k(\Z Q)$ (and hence all its quotients) satisfies Assumption \ref{assumption:Hom-finite and directed}.

For the next result, suppose $Q$ has no frozen vertices. Let $kQ$ be the path algebra of $Q$ and consider the bounded derived category $\Dc^b(\modcat kQ)$ of the category $\modcat kQ$ of finite-dimensional right $kQ$-modules. It is a Krull-Schmidt triangulated category. In particular, it can be completely described by its full subcategory $\ind(\Dc^b(\modcat kQ))$ of indecomposable objects. We have the following result due to Happel.

\begin{thm}[{\cite[Proposition I.5.6]{Happel88}}]\label{thm:Happel embedding}
    There is a fully faithful functor
    \[
    H: k(\Z Q) \longrightarrow \ind(\Dc^b(\modcat kQ)).
    \]
    sending the vertex $(i,0)$ to the indecomposable projective $kQ$-module associated with the vertex $i \in Q_0$. It is an equivalence if and only if $Q$ is a \emph{Dynkin quiver}, that is, its underlying graph is a disjoint union of $\mathsf{ADE}$ Dynkin diagrams.
\end{thm}

Through the functor above, the automorphism $\tau$ corresponds to the Auslander--Reiten translation on $\Dc^b(\modcat kQ)$, which we denote by the same symbol. In the Dynkin case, we can use the equivalence to transport the suspension functor $\Sigma$ of $\Dc^b(\modcat kQ)$ to an automorphism of the mesh category $k(\Z Q)$. In particular, it induces a bijection $\Sigma: (\Z Q)_0 \to (\Z Q)_0$.

\section{The category of split filtrations}\label{section:split filtrations}

In this section, we present the main definition of this paper: the category of split filtrations. We show that such a category is equivalent to the module category over a certain triangular matrix category. We then use this description and Section \ref{section:eight functors} to study natural functors relating split filtrations of different lengths.

For what follows, we fix a field $k$ and an integer $n \geq 1$. Let $\Ac$ be a small $k$-category and $\Bc$ a (not necessarily full) $k$-subcategory of $\Ac$. We always assume that $\Bc$ contains all objects of $\Ac$.

\subsection{Definition and main characterization}\label{section:general definition of splitting category} With the setup above, we give the following definition.

\begin{defn}\label{defn:definition of category of split filtrations}
We define $\Filt^n_{\Bc}(\Ac)$ as the category of filtrations of length $n$ in $\Mod\Ac$ together with a splitting over $\Mod\Bc$. More precisely, an object of $\Filt^n_{\Bc}(\Ac)$ is an $\Ac$-module $M$ together with a filtration
\[
0 = M_0 \subseteq M_1 \subseteq \dotsb \subseteq M_{n-1} \subseteq M_n = M
\]
by $\Ac$-submodules and $\Bc$-linear maps $r^M_i: M_i \to M_{i-1}$ which are retractions for the inclusions $M_{i-1} \to M_i$ ($1 \leq i \leq n$). A morphism in $\Filt^n_{\Bc}(\Ac)$ is a morphism $f: M \to N$ of $\Ac$-modules such that $f(M_i) \subseteq N_i$ and $r^N_i\circ f_i = f_{i-1} \circ r^M_i$ as maps of $\Bc$-modules, for all $1 \leq i \leq n$, where $f_i: M_i \to N_i$ denotes the restriction of $f$. The composition law is given by the composition of morphisms of $\Ac$-modules.
\end{defn}

\begin{remark}\label{rem:cofiltrations}
    The category above is equivalent to the category $\CoFilt_{\Bc}^n(\Ac)$ of ``cofiltrations" of length $n$ in $\Mod\Ac$ together with a splitting over $\Mod\Bc$, that is, the analogous category where each object is a sequence of $n-1$ epimorphisms
    \[\begin{tikzcd}[column sep = large]
        {N_1} & {N_2} & \dotsb & {N_n}
        \arrow["{f_1}"', two heads, from=1-1, to=1-2]
        \arrow["{f_2}"', two heads, from=1-2, to=1-3]
        \arrow["{f_{n-1}}"', two heads, from=1-3, to=1-4]
        \arrow["s_1"{description}, shift right, bend right=20, dashed, from=1-2, to=1-1]
        \arrow["s_2"{description}, shift right, bend right=20, dashed, from=1-3, to=1-2]
        \arrow["s_{n-1}"{description}, shift right, bend right=20, dashed, from=1-4, to=1-3]
    \end{tikzcd}\]
    equipped with $\Bc$-linear sections $s_i$ and with morphisms respecting this structure. The equivalence sends an object $M \in \Filt_{\Bc}^n(\Ac)$ to the sequence
    \[
    M =M/M_0 \twoheadrightarrow M/M_1 \twoheadrightarrow \dotsb \twoheadrightarrow M/M_{n-1}. 
    \]
    Each epimorphism has a canonical $\Bc$-linear section induced by the splitting of $M$.
\end{remark}

The goal of this subsection is to prove the following result. 

\begin{prop}\label{prop:characterization of Filt as a module category}
The category $\Filt_{\Bc}^n(\Ac)$ is equivalent to $\Mod\Tc_{\Bc}^n(\Ac)$ for a certain $k$-category $\Tc_{\Bc}^n(\Ac)$ (see Lemma \ref{lemma:explicit description of the triangular category} below for the definition).
\end{prop}

\begin{remark}
  This result is not true if we remove the assumption that $\Bc$ contains all objects of $\Ac$. For example, for $n = 2$ and $\Bc$ the empty subcategory of $\Ac$, the category $\Filt^n_{\Bc}(\Ac)$ is equivalent to the category of short exact sequences of $\Ac$-modules, which is not even abelian if $\Ac$ is not empty (see \cite[XII.6, p. 375]{MacLane63}).
\end{remark}

\begin{lemma}
The category $\Filt_{\Bc}^n(\Ac)$ is $k$-linear, cocomplete and abelian.
\end{lemma}

\begin{proof}
    The only part that is not immediate is the existence of cokernels. Let $f: M \to N$ be a morphism in $\Filt^n_{\Bc}(\Ac)$. If $f_i: M_i \to N_i$ ($1 \leq i \leq n$) denote the restrictions of $f$, then the filtration of $N$ induces maps $\mu_i: \coker f_{i-1} \to \coker f_i$, while the retractions $r^N_i$ induce $\Bc$-linear retractions of the $\mu_i$. Since $\Ac$ and $\Bc$ have the same set of objects, the existence of such a $\Bc$-linear retraction is enough to guarantee that $\mu_i$ is a monomorphism of $\Ac$-modules. Therefore, $\coker f$ comes equipped with a canonical filtration with splitting over $\Bc$, thus giving an object of $\Filt^n_{\Bc}(\Ac)$ with the desired universal property.
\end{proof}

Thus, one way to prove Proposition \ref{prop:characterization of Filt as a module category} is to find a set of compact projective objects in $\Filt_{\Bc}^n(\Ac)$ that generates the whole category. It turns out that such a collection of objects is easier to describe in the equivalent category $\CoFilt_{\Bc}^n(\Ac)$ (see Remark \ref{rem:cofiltrations}), so we will work with it until the end of this subsection.

Let $\mu: \Ac\otimes_{\Bc}\Ac \to \Ac$ be the morphism of $\Ac$-$\Ac$-bimodules induced by the composition of the category $\Ac$. For $1 \leq k \leq n$, consider the cofiltration of $\Ac$-bimodules
\[\begin{tikzcd}
	\Ac^{\otimes_{\Bc}k} & \Ac^{\otimes_{\Bc}k-1} & \dotsb & \Ac\otimes_{\Bc}\Ac & \Ac.
	\arrow["\id \otimes \mu", from=1-1, to=1-2]
	\arrow["\id \otimes \mu", from=1-2, to=1-3]
	\arrow["\id \otimes \mu", from=1-3, to=1-4]
	\arrow["\mu", from=1-4, to=1-5]
\end{tikzcd}\]
It has length $k$. Given an object $x \in \Ac$, we can specialize the bimodules above to get a cofiltration $\Pc^k_x$ of right $\Ac$-modules
\[\begin{tikzcd}
	x^{\wedge} \otimes_{\Bc}\Ac^{\otimes_{\Bc}k-1} & x^{\wedge}\otimes_{\Bc}\Ac^{\otimes_{\Bc}k-2} & \dotsb & x^{\wedge}\otimes_{\Bc}\Ac & x^{\wedge},
	\arrow["\id \otimes \mu", from=1-1, to=1-2]
	\arrow["\id \otimes \mu", from=1-2, to=1-3]
	\arrow["\id \otimes \mu", from=1-3, to=1-4]
	\arrow["\mu", from=1-4, to=1-5]
\end{tikzcd}\]
where $x^{\wedge} = \Ac(-,x)$ is the representable $\Ac$-module associated with $x$. Each one of the maps above has a canonical $\Bc$-linear section which sends a pure tensor $a_1 \otimes \dotsb \otimes a_l$ to $a_1 \otimes \dotsb \otimes a_l \otimes \id$. Extending it to a cofiltration of length $n$ by putting zeros at the end, we view $\Pc^k_x$ as an object of $\CoFilt_{\Bc}^n(\Ac)$.

\begin{lemma}\label{lemma:compact projective generators}
For $1 \leq k \leq n$ and $x \in \Ac$, the object $\Pc^k_x$ is compact and projective in $\CoFilt_{\Bc}^n(\Ac)$. Moreover, the direct sum of these objects as we vary $k$ and $x$ is a generator for $\CoFilt_{\Bc}^n(\Ac)$.
\end{lemma}

\begin{proof}
    We claim that the functor $\Hom(\Pc^k_x,-)$ is isomorphic to the functor that sends a cofiltration
    \[\mathcal{N} = \begin{tikzcd}[column sep = large]
        {N_1} & {N_2} & \dotsb & {N_n}
        \arrow["{f_1}"', two heads, from=1-1, to=1-2]
        \arrow["{f_2}"', two heads, from=1-2, to=1-3]
        \arrow["{f_{n-1}}"', two heads, from=1-3, to=1-4]
        \arrow["s_1"{description}, shift right, bend right=20, dashed, from=1-2, to=1-1]
        \arrow["s_2"{description}, shift right, bend right=20, dashed, from=1-3, to=1-2]
        \arrow["s_{n-1}"{description}, shift right, bend right=20, dashed, from=1-4, to=1-3]
    \end{tikzcd}\]
    to the evaluation of the module $\ker f_k$ at the object $x$ (for $k = n$, we set $f_n: N_n \to N_{n+1}=0$ as the zero morphism). This claim implies the lemma. Indeed, the latter functor commutes with arbitrary direct sums, so $\Pc^k_x$ is compact. The snake lemma shows that it is exact, so $\Pc^k_x$ is projective. Moreover, $\mathcal{N}$ is zero if and only if $\ker f_k = 0$ for all $1 \leq k \leq n$, implying that the direct sum of the $\Pc^k_x$ is a generator.

    Let us prove our claim. Part of the data of a morphism $\Pc^k_x \to \mathcal{N}$ is a commutative diagram
    \[\begin{tikzcd}
	{x^{\wedge}} & 0 \\
	{N_k} & {N_{k+1}}.
	\arrow[from=1-1, to=1-2]
	\arrow["\eta", from=1-1, to=2-1]
	\arrow[from=1-2, to=2-2]
	\arrow["{f_k}", from=2-1, to=2-2]
    \end{tikzcd}\]
    By the Yoneda lemma, the morphism on the left corresponds to an element $\eta \in N_k(x)$, and the diagram commutes if and only if $\eta$ lies in the kernel of $f_k$. Reciprocally, let us show that a commutative diagram as the one above uniquely lifts to a morphism $\Pc^k_x \to \mathcal{N}$. Uniqueness follows from the fact that the image of the $\Bc$-linear section
    \[
    x^{\wedge}\otimes_{\Bc}\Ac^{\otimes_{\Bc}l-1} \to x^{\wedge}\otimes_{\Bc}\Ac^{\otimes_{\Bc}l}
    \]
    in the definition of $\Pc^k_x$ generates $x^{\wedge}\otimes_{\Bc}\Ac^{\otimes_{\Bc}l}$ as an $\Ac$-module. Hence, since a morphism $\Pc^k_x \to \mathcal{N}$ has to be compatible with the sections, it is determined by its component $x^{\wedge} \to N_k$. This last observation forces us to define the desired morphism as follows: its component $x^{\wedge}\otimes_{\Bc}\Ac \to N_{k-1}$ sends a pure tensor $a_1 \otimes a_2$ to
    \[
    s_{k-1}(\eta \cdot a_1) \cdot a_2,
    \]
    its component $x^{\wedge}\otimes_{\Bc}\Ac^{\otimes_{\Bc}2} \to N_{k-2}$ sends a pure tensor $a_1 \otimes a_2 \otimes a_3$ to
    \[
    s_{k-2}(s_{k-1}(\eta \cdot a_1) \cdot a_2) \cdot a_3,
    \]
    and so on. One checks that this is indeed well-defined and yields a morphism from $\Pc^k_x$ to $\mathcal{N}$.
\end{proof}

We define a category $\Tc = \Tc^n_{\Bc}(\Ac)$ as follows. Its set of objects is the disjoint union of $n$ copies of the set of objects of $\Ac$. For $x \in \Ac$, we denote the corresponding object in the $i$-th copy by $x^i$. For $x,y \in \Ac$ and $1 \leq i,j \leq n$, we put
\[
\Tc(x^i,y^j) = \Hom_{\CoFilt_{\Bc}^n(\Ac)}(\Pc^i_x,\Pc^j_y).
\]
The composition law comes from that of $\CoFilt_{\Bc}^n(\Ac)$. By a classical result of Freyd (see \cite[Theorem 3.1]{Mitchell72}), Lemma \ref{lemma:compact projective generators} implies that $\CoFilt_{\Bc}^n(\Ac)$ is equivalent to $\Mod\Tc$. The equivalence sends $\mathcal{N}\in \CoFilt_{\Bc}^n(\Ac)$ to the $\Tc$-module $F_{\mathcal{N}}$ defined on objects by
\[
F_{\mathcal{N}}(x^i) = \Hom_{\CoFilt_{\Bc}^n(\Ac)}(\Pc^i_x,\mathcal{N}),
\]
with the evident action on morphisms. 

Since $\CoFilt_{\Bc}^n(\Ac)$ is equivalent to $\Filt_{\Bc}^n(\Ac)$, this completes the proof of Proposition \ref{prop:characterization of Filt as a module category}. However, our definition of the category $\Tc$ is a bit abstract. We will now provide a more concrete description of it.

For $k \geq 2$, define the $\Ac$-bimodule $X_k$ to be the kernel of the morphism
\[
\id^{\otimes k-2} \otimes \mu: \Ac^{\otimes_{\Bc}k} \longrightarrow \Ac^{\otimes_{\Bc}k-1}
\]
given by the composition of the last two factors. We define $X_1$ to be the regular bimodule $\Ac$. For $k,l \geq 1$, we remark that we have the morphism of bimodules
\[
\mu_{kl}: X_k \otimes_{\Ac} X_l \longrightarrow X_{k+l-1}
\]
that multiplies the last factor of a pure tensor in $X_k$ with the first factor of a pure tensor in $X_l$. In particular, $\mu_{11} = \mu$ is the map induced by the composition of $\Ac$.

\begin{lemma}\label{lemma:explicit description of the triangular category}
The category $\Tc = \Tc^n_{\Bc}(\Ac)$ is isomorphic to the triangular matrix category
\[
    \begin{pmatrix}
    \Ac & 0 & 0 & \dotsb & 0\\
    X_2 & \Ac & 0 & \dotsb & 0\\
    X_3 & X_2 & \Ac & \dotsb & 0\\
    \vdots & \vdots & \vdots & \ddots & \vdots\\
    X_n & X_{n-1} & X_{n-2} & \dotsb & \Ac
\end{pmatrix}.
\]
More precisely, $\Tc(x^j,y^i)$ vanishes for $i < j$ and we have an isomorphism
\[
\Tc(x^j,y^i) \cong X_{i-j+1}(x,y)
\]
for $i \geq j$ in such a way that, under this identification, the composition law in $\Tc$ is given by the maps $\mu_{kl}: X_k \otimes_{\Ac} X_l \longrightarrow X_{k+l-1}$.
\end{lemma}

\begin{proof}
The description of $\Tc(x^j,y^i)$ follows immediately from the characterization of the functor $\Hom(\Pc^j_x,-)$ given in the proof of Lemma \ref{lemma:compact projective generators}. The verification of the last assertion is left to the reader.
\end{proof}

\subsection{The canonical functors}\label{section:canonical functors} Let us describe some canonical functors relating the categories $\Filt_{\Bc}^n(\Ac)$ for various $n$. For $m \geq n \geq 1$, we have the functors $\sub^{m,n},\quot^{m,n} : \Filt^m_{\Bc}(\Ac) \to \Filt^n_{\Bc}(\Ac)$ defined on an object $M$ by
\[
\sub^{m,n}(M) = M_n \quad \textrm{and} \quad \quot^{m,n}(M) = M/M_{m-n}.
\]
Both modules above inherit the filtration and the $\Bc$-linear splitting from $M$. In the other direction, we have the functors $i_{\sub}^{n,m},i_{\quot}^{n,m}: \Filt^n_{\Bc}(\Ac) \to \Filt^m_{\Bc}(\Ac)$ which take an object $M$ and trivially extend its filtration to length $m$ by repeating $M$ at the end or $0$ at the beginning, respectively. Finally, we denote by $\tot^n: \Filt^n_{\Bc}(\Ac) \to \Mod\Ac$ the functor which forgets the filtration and the splitting. These functors are all exact.

Most of these functors can be realized in the setting of Section \ref{section:eight functors}. For this purpose, via Proposition \ref{prop:characterization of Filt as a module category}, we identify $\Filt^n_{\Bc}(\Ac)$ with the category of modules over $\Tc^n = \Tc^n_{\Bc}(\Ac)$. Fix $m,n \geq 1$. The characterization in Lemma \ref{lemma:explicit description of the triangular category} allows us to view $\Tc^{m+n}$ as the triangular matrix category
\[
\begin{pmatrix}
    \Tc^m & 0\\
    X_{mn} & \Tc^n
\end{pmatrix},
\]
where the rectangular block
\[
X_{mn} = \begin{pmatrix}
    X_{m+1} & X_m & \dotsb & X_2\\
    X_{m+2} & X_{m+1} & \dotsb & X_3\\
    \vdots & \vdots &  & \vdots\\
    X_{m+n} & X_{m+n-1} & \dotsb & X_{n+1}
\end{pmatrix}
\]
can be naturally viewed as a $\Tc^n$-$\Tc^m$-bimodule using the maps $\mu_{kl}$ from Section \ref{section:general definition of splitting category}. This decomposition gives rise to the eight functors in Section \ref{section:eight functors}.

\begin{lemma}\label{lemma:canonical functors as gluing data}
If we decompose $\Tc^{m+n}$ as the triangular matrix category above and consider the corresponding diagram as in Section \ref{section:eight functors}, then the functors $i_*$, $i^{-1}$, $j_!$ and $j^{-1}$ are isomorphic to $i_{\quot}^{n,m+n}$, $\quot^{m+n,n}$, $i_{\sub}^{m,m+n}$ and $\sub^{m+n,m}$, respectively. 
\end{lemma}

\begin{proof}
    By composing the equivalence $\Filt^{m+n}_{\Bc}(\Ac) \to \CoFilt^{m+n}_{\Bc}(\Ac)$ from Remark \ref{rem:cofiltrations} with the equivalence $\CoFilt^{m+n}_{\Bc}(\Ac) \to \Mod\Tc^{m+n}_{\Bc}(\Ac)$ from the proof of Proposition \ref{prop:characterization of Filt as a module category}, we see that a split filtration
    \[
    0 = M_0 \subseteq M_1 \subseteq \dotsb \subseteq M_{m+n} = M
    \]
    in $\Filt^{m+n}_{\Bc}(\Ac)$ corresponds to the $\Tc_{\Bc}^{m+n}(\Ac)$-module
    \[
    (M_1,M_2/M_1,\dots,M_{m+n}/M_{m+n-1}, \varphi_{ij}).
    \]
    Here, the splitting over $\Bc$ induces the direct sum decomposition $M \cong M_1 \oplus M_2/M_1 \oplus \dotsb \oplus M_{m+n}/M_{m+n-1}$ (as $\Bc$-modules) and the strucutral maps $\varphi_{ij}$ encode the $\Ac$-module structure on $M$ in terms of these summands. The lemma then follows without much difficulty from the definition of the functors above.
\end{proof}

The main result of this subsection is the following.

\begin{prop}\label{prop:sub/quot preserve projective/injective}
Suppose that the restriction of the regular $\Ac$-$\Ac$-bimodule to $\Mod\Bc$ is projective on both sides. Then the functors $\sub^{m+n,m}$ and $\quot^{m+n,n}$ preserve projective and injective objects for all $m,n \geq 1$.
\end{prop}

\begin{remark}
If $\Bc$ is the discrete $k$-subcategory of $\Ac$ generated by all identity morphisms, then any $\Ac$-$\Ac$-bimodule is projective on both sides as a $\Bc$-module, so that we can apply the result above.
\end{remark}

To prove the proposition, let us focus first on the case of the functor $\sub^{m+n,m}$. By the identity
\[
\sub^{j,i}\circ \sub^{k,j} = \sub^{k,i}
\]
valid for $1 \leq i \leq j \leq k$, we can assume that $n = 1$. It is enough to show that both adjoints of $\sub^{m+1,m}$ are exact. By Lemma \ref{lemma:canonical functors as gluing data}, the left adjoint is $i_{\sub}^{m,m+1}$, which is indeed exact. By the same lemma, the right adjoint is the functor $j_*: \Mod\Tc^m \to \Mod\Tc^{m+1}$. It is exact if and only if the functor $\Hom_{\Tc^m}(X_{m1},-): \Mod\Tc^m \to \Mod\Ac$ is exact, that is, if and only if $X_{m1}$ is projective on the right. Before we show that this is the case, we need some preparation.

Since $\Ac$ is projective on both sides as a $\Bc$-module, we have that $\Ac^{\otimes_{\Bc}k}$ is projective on both sides as an $\Ac$-module for all $k \geq 1$. Consequently, for $k \geq 2$, the exact sequence
\[\begin{tikzcd}
	0 & {X_k} & {\Ac^{\otimes_{\Bc}k}} & {\Ac^{\otimes_{\Bc}k-1}} & 0
	\arrow[from=1-1, to=1-2]
	\arrow[from=1-2, to=1-3]
	\arrow["{\id \otimes \mu}", from=1-3, to=1-4]
	\arrow[from=1-4, to=1-5]
\end{tikzcd}\]
defining $X_k$ splits and $X_k$ is also projective on both sides. Thus, if we apply $-\otimes_{\Ac}X_l$ ($l \geq 2$) to the exact sequence above, we get the exact sequence
\[\begin{tikzcd}
	0 & {X_k \otimes_{\Ac} X_l} & {\Ac^{\otimes_{\Bc}k-1} \otimes_{\Bc} X_l} & {\Ac^{\otimes_{\Bc}k-2}\otimes_{\Bc} X_l} & 0,
	\arrow[from=1-1, to=1-2]
	\arrow[from=1-2, to=1-3]
	\arrow["{\eta_{kl}}", from=1-3, to=1-4]
	\arrow[from=1-4, to=1-5]
\end{tikzcd}\]
where $\eta_{kl}$ multiplies the last factor of a pure tensor in $\Ac^{\otimes_{\Bc}k-1}$ with the first factor of a pure tensor in $X_l$. On the other hand, if we tensor the first exact sequence $t$ times on the left by $\Ac$ over $\Bc$, we get an exact sequence where the last map is still of the form $\id \otimes \mu$, so we deduce that $\Ac^{\otimes_{\Bc}t}\otimes_{\Bc}X_k \cong X_{k+t}$. Therefore, we may identify the second exact sequence above with the sequence
\begin{equation}\label{eq:exact sequence of bimodules}
\begin{tikzcd}
	0 & {X_k \otimes_{\Ac} X_l} & {X_{k+l-1}} & {X_{k+l-2}} & 0.
	\arrow[from=1-1, to=1-2]
	\arrow["{\mu_{kl}}", from=1-2, to=1-3]
	\arrow["{\eta_{kl}}", from=1-3, to=1-4]
	\arrow[from=1-4, to=1-5]
\end{tikzcd}
\end{equation}
We can now continue the proof.

\begin{lemma}\label{lemma:bimodule is projective on a side}
    With the hypothesis of Proposition \ref{prop:sub/quot preserve projective/injective}, the bimodule $X_{m1}$ is projective on the right for all $m \geq 1$.
\end{lemma}

\begin{proof}
    For $1 \leq k \leq m$, the $1 \times m$ ``row vector''
    \[
    Y_k = \begin{pmatrix}
        X_{k+1} & \dotsb & X_3 & X_2 & 0 & \dotsb & 0
    \end{pmatrix}
    \]
    is naturally an $\Ac$-$\Tc^m$-bimodule via matrix multiplication. Notice that $Y_m = X_{m1}$. We will prove by induction on $k$ that $Y_k$ is projective on the right. For $k=1$ and an object $x \in \Ac$, the right $\Tc^m$-module $Y_1(-,x)$ is given by $i_{\sub}^{1,m}(X_2(-,x))$. By Lemma \ref{lemma:canonical functors as gluing data}, the right adjoint of $i_{\sub}^{1,m}$ is $\sub^{m,1}$, which is exact, hence $i_{\sub}^{1,m}$ preserves projective objects. By the discussion above, $X_2(-,x)$ is a projective $\Ac$-module, thus $Y_1(-,x)$ is a projective $\Tc^n$-module.

    Suppose $k \geq 2$. Consider the $1 \times m$ ``row vector''
    \[
    Z_k = \begin{pmatrix}
        X_2 \otimes_{\Ac} X_k & \dotsb & X_2 \otimes_{\Ac} X_2 & X_2 \otimes_{\Ac}{\Ac} & 0 & \dotsb & 0
    \end{pmatrix}.
    \]
    It is an $\Ac$-$\Tc^n$-bimodule where the structural maps of the $\Tc^n$-module structure are of the form $\id_{X_2} \otimes \mu_{ij}$. We claim that it is projective on the right. Let $i_!: \Mod\Ac \to \Mod\Tc^k$ be the functor of Section \ref{section:eight functors} arising from the decomposition
    \[
    \Tc^k = \begin{pmatrix}
        \Tc^{k-1} & 0\\
        X_{k-1,1} & \Ac
    \end{pmatrix}.
    \]
    Notice that $Z_k(-,x) = i_{\sub}^{k,n}(i_!(X_2(-,x)))$ for all $x \in \Ac$. Since the right adjoint of $i_!$ is $\quot^{1,k}$, which is exact, this functor preserves projectives. But $i_{\sub}^{k,n}$ also preserves projectives and $X_2(-,x)$ is a projective $\Ac$-module, whence we deduce our claim.

    We will prove the existence of an exact sequence
    \[\begin{tikzcd}
	0 & {Z_k} & {Y_k} & {Y_{k-1}} & 0.
	\arrow[from=1-1, to=1-2]
	\arrow[from=1-2, to=1-3]
	\arrow[from=1-3, to=1-4]
	\arrow[from=1-4, to=1-5]
    \end{tikzcd}\]
    Taking into account the induction hypothesis and the previous paragraph, this will imply that $Y_k$ is projective, concluding the proof. The exact sequence (\ref{eq:exact sequence of bimodules}) induces exact sequences between the components of the bimodules above, so we only need to check that the maps $\mu_{2,i}$ and $\eta_{2,i}$ are compatible with the $\Tc^n$-module structures. In other words, we have to check the commutativity of the diagrams (\ref{eq:morphisms in the comma category}) in Remark \ref{rem:triangular matrix categories of arbitrary order}. This follows from the diagrams
    \[\begin{tikzcd}
        {X_2 \otimes_{\Ac} X_j} & {X_2 \otimes_{\Ac} X_j} & {0}\\
    	{X_2 \otimes_{\Ac} X_j} & {X_{j+1}} & {X_j}
    	\arrow[equal, from=1-1, to=2-1]
    	\arrow[equal, from=1-1, to=1-2]
    	\arrow["{\mu_{2,j}}"', from=2-1, to=2-2]
    	\arrow["{\mu_{2,j}}"', from=1-2, to=2-2]
        \arrow[from=1-2, to=1-3]
        \arrow["{\eta_{2,j}}"', from=2-2, to=2-3]
        \arrow[from=1-3, to=2-3]
    \end{tikzcd}\]
    and
    \[\begin{tikzcd}
        {X_2 \otimes_{\Ac} X_i \otimes_{\Ac} X_j} & {X_{i+1} \otimes_{\Ac} X_j} & {X_i \otimes_{\Ac}X_j}\\
    	{X_2 \otimes_{\Ac} X_{i+j-1}} & {X_{i+j}} & {X_{i+j-1}}
    	\arrow["{\id \otimes \mu_{ij}}"', from=1-1, to=2-1]
    	\arrow["{\mu_{2,i} \otimes \id}", from=1-1, to=1-2]
    	\arrow["{\mu_{2,i+j-1}}"', from=2-1, to=2-2]
    	\arrow["{\mu_{i+1,j}}"', from=1-2, to=2-2]
        \arrow["{\eta_{2,i} \otimes \id}", from=1-2, to=1-3]
        \arrow["{\eta_{2,i+j-1}}"', from=2-2, to=2-3]
        \arrow["{\mu_{ij}}", from=1-3, to=2-3]
    \end{tikzcd}\]
    which are commutative for $i,j \geq 2$ as one can easily check.
\end{proof}

This completes the proof of Proposition \ref{prop:sub/quot preserve projective/injective} for the functor $\sub^{m+n,m}$. For the functor $\quot^{m+n,n}$, a similar argument works, and the proof is reduced to showing that $X_{1n}$ is projective on the left for all $n \geq 1$. This can be proved as in the proof of Lemma \ref{lemma:bimodule is projective on a side}.

\begin{remark}\label{rem:sub/quot preserve fg proj/inj}
Using the adjunction $\quot^{m+n,n} \dashv i_{\quot}^{n,m+n}$, one can see that $\quot^{m+n,n}$ sends the free module $(x^i)^{\wedge}$ in $\Mod\Tc^{m+n}$ (for $x \in \Ac$ and $1 \leq i \leq m+n$) to the free module $(x^{i-m})^{\wedge}$ in $\Mod\Tc^n$ (if $i > m$) or to zero (if $i \leq m$). In general, it is not clear if we have an analogous property for cofree modules. Nevertheless, if we assume the hypothesis of Proposition \ref{prop:sub/quot preserve projective/injective} and that the endomorphism algebra of every object in $\Ac$ is local, then $\quot^{m+n,n}$ sends a cofree module to a product of cofree modules. Indeed, we may suppose $m = 1$. In this case, the left adjoint of $\quot^{n+1,n}$ can be described using the bimodule $X_{1n}$, which is projective on the left. By Kaplansky's theorem \cite{Kaplansky58}, any specialization of $X_{1n}$ to a left module is hence a direct sum of free modules. Knowing this, it is not hard to deduce our claim. Dually, under the same hypotheses on $\Ac$, $\sub^{m+n,m}$ sends a cofree module to another cofree module or to zero, and a free module to a coproduct of free modules.
\end{remark}

\section{Split filtrations for the singular Nakajima category}\label{section:generalizing KS}

We study the category of split filtrations over the singular Nakajima category $\mathcal{S}$, as defined by Keller--Scherotzke in \cite{KellerScherotzke16}. Its properties are similar to those of $\mathcal{S}$, and we adapt some arguments in \cite{KellerScherotzke16} to our setting without much difficulty. The initial results hold for an arbitrary finite acyclic quiver, but we specialize to the case of a Dynkin quiver starting from Section \ref{section:Gorenstein modules}. Most of the results hold for an arbitrary base field $k$, but we will suppose that $k = \C$ whenever we use the varieties defined in Section \ref{section:link to quiver varieties}.

\subsection{Alternative definition using mesh categories}

Let $Q$ be a finite acyclic quiver (without frozen vertices) and $n \geq 1$ a positive integer. The \emph{$n$-framed quiver} $Q\fr{n}$ is obtained from $Q$ by adding, for each vertex $i \in Q_0$, $n$ new vertices $i_1, i_2, \dots, i_n$, and $n$ new arrows $i \to i_r$ ($1 \leq r \leq n$). The new vertices are considered to be frozen.

Let $\Z Q\fr{n}$ be the repetition quiver of $Q\fr{n}$. For example, if $Q = 1 \to 2$ and $n=2$, then $\Z Q\fr{2}$ is the infinite quiver
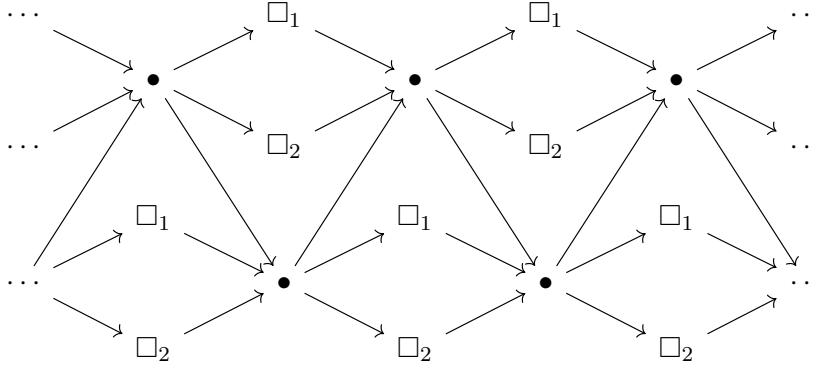
\[\begin{tikzcd}[column sep=small, row sep=small]
	\dotsb &&&& \square_1 &&&& \square_1 &&&& \dotsb \\
	&& \bullet &&&& \bullet &&&& \bullet \\
	\dotsb &&&& \square_2 &&&& \square_2 &&&& \dotsb \\
	&& \square_1 &&&& \square_1 &&&& \square_1 \\
	\dotsb &&&& \bullet &&&& \bullet &&&& \dotsb \\
	&& \square_2 &&&& \square_2 &&&& \square_2
	\arrow[from=1-1, to=2-3]
	\arrow[from=1-5, to=2-7]
	\arrow[from=1-9, to=2-11]
	\arrow[from=2-3, to=1-5]
	\arrow[from=2-3, to=3-5]
	\arrow[from=2-3, to=5-5]
	\arrow[from=2-7, to=1-9]
	\arrow[from=2-7, to=3-9]
	\arrow[from=2-7, to=5-9]
	\arrow[from=2-11, to=1-13]
	\arrow[from=2-11, to=3-13]
	\arrow[from=2-11, to=5-13]
	\arrow[from=3-1, to=2-3]
	\arrow[from=3-5, to=2-7]
	\arrow[from=3-9, to=2-11]
	\arrow[from=4-3, to=5-5]
	\arrow[from=4-7, to=5-9]
	\arrow[from=4-11, to=5-13]
	\arrow[from=5-1, to=2-3]
	\arrow[from=5-1, to=4-3]
	\arrow[from=5-1, to=6-3]
	\arrow[from=5-5, to=2-7]
	\arrow[from=5-5, to=4-7]
	\arrow[from=5-5, to=6-7]
	\arrow[from=5-9, to=2-11]
	\arrow[from=5-9, to=4-11]
	\arrow[from=5-9, to=6-11]
	\arrow[from=6-3, to=5-5]
	\arrow[from=6-7, to=5-9]
	\arrow[from=6-11, to=5-13]
\end{tikzcd}\]
where we represent by $\square_r$ the frozen vertices of the form $(i_r,p)$ for $i \in Q_0$. For a non-frozen vertex $x = (i,p)$, we put $\sigma_r(x) = (i_r,p-1)$, so that $\sigma_1(x),\sigma_2(x),\dots,\sigma_n(x)$ are the frozen vertices immediately preceding $x$. Similarly, we denote by $\sigma^{-1}_1(x),\sigma^{-1}_2(x),\dots,\sigma^{-1}_n(x)$ the frozen vertices immediately succeding $x$. When $n = 1$, we write $\sigma_1 = \sigma$ and $\sigma^{-1}_1 = \sigma^{-1}$.

\begin{defn}\label{defn:Nakajima categories of filtrations}
    Let $Q$ be a finite acyclic quiver and $n \geq 1$ a positive integer.
    \begin{itemize}
        \item The \emph{regular Nakajima category of split filtrations of length $n$} is the quotient $\Rc\filt{n}$ of the mesh category $k(\Z Q\fr{n})$ by the smallest $k$-ideal containing all morphisms from a frozen vertex of the form $(i_r,p)$ to a frozen vertex of the form $(j_s,q)$, where $i,j \in Q_0$ and $r > s$.
        \item The \emph{singular Nakajima category of split filtrations of length $n$} is the full subcategory $\Sc\filt{n}$ of $\Rc\filt{n}$ whose objects are the frozen vertices.
    \end{itemize}
    For $n = 1$, we denote $\Rc\filt{1}$ and $\Sc\filt{1}$ by $\Rc$ and $\Sc$, respectively.
\end{defn}

\begin{remark}
    When $n = 1$, the given set of relations on $k(\Z Q\fr{1})$ is empty, so that $\Rc = k(\Z Q\fr{1})$. Hence, the categories $\Rc$ and $\Sc$ are precisely the \emph{regular Nakajima category} and the \emph{singular Nakajima category} defined in \cite{KellerScherotzke16}.
\end{remark}

\begin{prop}\label{prop:modules over Sfilt are the same as Filt(S)}
The category $\Sc\filt{n}$ is isomorphic to the category $\Tc_{k\Sc_0}^n(\Sc)$ of Section \ref{section:general definition of splitting category}, where $k\Sc_0$ denotes the discrete $k$-subcategory of $\Sc$ generated by all identity morphisms.
\end{prop}

\begin{proof}
    By construction, both $\Sc\filt{n}$ and $\Tc_{k\Sc_0}^n(\Sc)$ are spectroids in the sense of \cite{GabrielRoiter}, that is, they are $\Hom$-finite $k$-categories where different objects are non-isomorphic and all endomorphism algebras are local. Thus, by \cite[Sections 2.5, 3.3]{GabrielRoiter}, it is enough to prove that $\Mod\Sc\filt{n}$ is equivalent to $\Mod\Tc^n_{k\Sc_0}(\Sc)$, which in turn is equivalent to $\Filt_{k\Sc_0}^n(\Sc)$ by Proposition \ref{prop:characterization of Filt as a module category}.

    It is not hard to check that a $k(\Z Q\fr{n})$-module has the same data as an $\Rc$-module $M: \Rc\op \to \Mod k$ together with a vector space decomposition $M(y) = M_1(y) \oplus \dotsb \oplus M_n(y)$ for every frozen vertex $y$ in $\Z Q\fr{1}$. If we replace $k(\Z Q\fr{n})$ by its subcategory of frozen vertices and $\Rc$ by $\Sc$, the analogous statement also holds. If we impose the additional relations given in Definition \ref{defn:Nakajima categories of filtrations}, we deduce that an $\Sc\filt{n}$-module corresponds to an $\Sc$-module $M: \Sc\op \to \Mod k$ together with a vector space decomposition $M(y) = M_1(y) \oplus \dotsb \oplus M_n(y)$ for every frozen vertex $y$ and such that the first $l$ terms of this decomposition determines a submodule of $M$ for all $1 \leq l \leq n$. In other words, the data of an $\Sc\filt{n}$-module is equivalent to that of a filtration of length $n$ in $\Mod\Sc$ with a splitting over $k\Sc_0$. One checks that this correspondence is functorial and gives the desired equivalence.
\end{proof}

\subsection{The link to graded quiver varieties}\label{section:link to quiver varieties} For this subsection, we assume $k = \C$ is the field of complex numbers. The categories $\Rc$ and $\Sc$ are related to graded quiver varieties as follows. Denote by $\Rc_0$ and $\Sc_0$ their set of objects. Let $w: \Sc_0 \to \N$ be a \emph{dimension vector}, that is, a function with finite support. We define $\Mf_0(w)$ as the affine variety of $\Sc$-modules $M$ such that $M(u) = k^{w(u)}$ for $u \in \Sc_0$. It is the Zariski closed subset of the affine space
    \[
    \prod_{u_1,u_2 \in \Sc_0}\Hom_k(\Sc(u_1,u_2),\Hom_k(k^{w(u_2)},k^{w(u_1)}))
    \]
    cut out by the equations encoding the composition of $\Sc$. By the proof in \cite[Theorem 2.4]{LeclercPlamondon13} (see also \cite[Section 2.3]{KellerScherotzke16}), it can be identified with the \emph{affine graded quiver variety} defined by Nakajima \cite{Nakajima11} (when $Q$ is bipartite) and by Qin \cite{Qin14} (see also \cite{KimuraQin14}).

    The category $\Rc$ is used to define the \emph{(nonsingular) graded quiver variety} $\Mf(v,w)$ of \cite{Nakajima11,Qin14,KimuraQin14} associated with dimension vectors $v: \Rc_0 \setminus \Sc_0 \to \N$ and $w: \Sc_0 \to \N$. To define it, we need an auxiliary construction. We say an $\Rc$-module is \emph{stable} if it does not contain non-zero submodules concentrated in non-frozen vertices. Consider the affine variety $\Mft(v,w)$ of stable $\Rc$-modules $M$ such that $M(x) = k^{v(x)}$ and $M(\sigma(x)) = k^{w(\sigma(x))}$ for $x \in (\Z Q)_0$. The group $G_v = \prod_{x \in (\Z Q)_0}\operatorname{GL}(k^{v(x)})$ acts on $\Mft(v,w)$ by base change on non-frozen vertices. We then define $\Mf(v,w)$ as the quotient $\Mft(v,w)/G_v$. It can be canonically equipped with the structure of a smooth quasi-projective variety. The restriction along the inclusion $\Sc \to \Rc$ induces a proper map $\pi: \Mf(v,w) \longrightarrow \Mf_0(w)$. Notice that both $\Mf(v,w)$ and $\Mf_0(w)$ carry an action of the group $G_w = \prod_{x \in (\Z Q)_0}\operatorname{GL}(k^{w(\sigma(x))})$ by base change on frozen vertices, and the map $\pi$ is $G_w$-equivariant. 

    In the non-graded setting, Nakajima defines in \cite{Nakajima01} his tensor product varieties. A version for graded quiver varieties associated with bipartite quivers can be found in \cite[Section 3.5]{Nakajima11}. We define them for arbitrary quivers by mimicking his construction as follows. Let $w_1,\dots,w_n: \Sc_0 \to \N$ be dimension vectors and write $w$ for their sum. Choose a direct sum decomposition $k^{w(u)} = k^{w_1(u)} \oplus \dotsb \oplus k^{w_n(u)}$ for each $u \in \Sc_0$. The \emph{$n$-fold affine graded tensor product variety} $\Tf_0(w_1,\dots,w_n)$ is the closed subvariety of $\Mf_0(w)$ of all modules $M$ such that, for all $1 \leq l \leq n$, the subspaces $k^{w_1(u)} \oplus \dotsb \oplus k^{w_l(u)} \subseteq k^{w(u)} = M(u)$ are invariant under the $\Sc$-module action. This subvariety is invariant under the action of the subgroup $G_{w_1} \times \dotsb \times G_{w_n}$ of $G_w$ determined by the direct sum decomposition.
    
    \begin{remark}
        We use a convention opposite to the one in \cite{Nakajima01}, where the dimension vectors are reversed. For example, in their definition of $\Tf_0(w_1,w_2)$, the subspaces $k^{w_2(u)}$ are the ones invariant under the module action (see the proof of Lemma 3.6 in loc. cit.).
    \end{remark}

    Our definition immediately identifies the points of $\Tf_0(w_1,\dots,w_n)$ with objects in $\Filt^n_{k\Sc_0}(\Sc)$. In fact, by Proposition \ref{prop:modules over Sfilt are the same as Filt(S)} and its proof, $\Tf_0(w_1,\dots,w_n)$ is isomorphic to the variety of $\Sc\filt{n}$-modules $M$ such that $M(\sigma_r(x)) = k^{w_r(\sigma(x))}$ for all $x \in (\Z Q)_0$ and $1 \leq r \leq n$. The action by the group $G_{w_1}\times\dotsb\times G_{w_n}$ corresponds to base change on the vertices. For a dimension vector $v: \Rc_0 \setminus \Sc_0 \to \N$, one also has the subvariety $\Tf(v;w_1,\dots,w_n)$ of $\Mf(v,w)$ defined by $\pi^{-1}(\Tf_0(w_1,\dots,w_n))$. It is not hard to show that it has a similar description as $\Mf(v,w)$, but where $\Rc$-modules are replaced by $\Rc\filt{n}$-modules.

    \begin{remark}
        Nakajima also considers the subvariety $\Tft_0(w_1,\dots,w_n)$ of $\Tf_0(w_1,\dots,w_n)$ consisting of points corresponding to split filtrations of $\Sc$-modules whose successive quotients are semisimple. Strictly speaking, this is the variety he uses to geometrically realize tensor products of standard modules over the quantum affine algebra. However, since standard modules themselves are tensor products of fundamental modules (see \cite[Corollary 6.13]{Nakajima01}), when working with this geometric realization we may restrict to the case where each dimension vector $w_i$ is supported at a single vertex, where it has value $1$. In this situation, $\Tft_0(w_1,\dots,w_n)$ coincides with $\Tf_0(w_1,\dots,w_n)$.
    \end{remark}

\subsection{Projective and injective (co)resolutions} In this section, we present projective resolutions and injective (co)resolutions for simple modules over $\Rc\filt{n}$ and $\Sc\filt{n}$.

\begin{lemma}\label{lemma:resolutions in Rfilt}
For each non-frozen vertex $x$ of $\Z Q\fr{n}$, we have minimal (co)resolutions of simple $\Rc\filt{n}$-modules:
\begin{enumerate}[(a)]
    \item $0 \longrightarrow \tau(x)^{\wedge} \longrightarrow \sigma_n(x)^{\wedge} \longrightarrow S_{\sigma_n(x)} \longrightarrow 0$,
    \item $0 \longrightarrow S_{\sigma_1(x)} \longrightarrow \sigma_1(x)^{\vee} \longrightarrow x^{\vee} \longrightarrow 0$,
    \item $0 \longrightarrow \tau(x)^{\wedge} \longrightarrow \bigoplus_{y \to x}y^{\wedge} \longrightarrow x^{\wedge} \longrightarrow S_x \longrightarrow 0$,
    \item $0 \longrightarrow S_x \longrightarrow x^{\vee} \longrightarrow \bigoplus_{x \to y}y^{\vee} \longrightarrow \tau^{-1}(x)^{\vee} \longrightarrow 0$,
\end{enumerate}
where the direct sums vary over all arrows in $\Z Q\fr{n}$ ending or starting at $x$, respectively.
\end{lemma}

\begin{proof}
For (a), the only non-trivial verification is that the map $\tau(x)^{\wedge} \to \sigma_n(x)^{\wedge}$ is a monomorphism. This follows since there are no relations in the definition of $\Z Q\fr{n}$ and $\Rc\filt{n}$ ending with the arrow $\tau(x) \to \sigma_n(x)$. We also deduce that the first map in (c) is a monomorphism. To verify the exactness of (c) at the direct sum, observe that, if a non-trivial linear combination of morphisms becomes zero after composing with the arrows $y \to x$, then this composition must be a consequence of the mesh relation $r_x$ from (\ref{eq:mesh relation}). Therefore, such a linear combination is in the image of the first map in (c). The proofs for (b) and (d) are similar.
\end{proof}

\begin{cor}\label{cor:RHom vanishes}
    Let $x \in (\Z Q\fr{n})_0$ be a non-frozen vertex. Then $\RHom(S_x, M)$ is acyclic if $M$ is a coproduct of free $\Rc\filt{n}$-modules associated with frozen vertices. Dually, $\RHom(N, S_x)$ is acyclic if $N$ is a product of cofree $\Rc\filt{n}$-modules associated with frozen vertices.
\end{cor}

\begin{proof}
Using the resolutions (c) and (d) of Lemma \ref{lemma:resolutions in Rfilt}, one can compute that the dual of $\RHom(S_x, M)$ is isomorphic to a shift of $\RHom(M,S_{\tau(x)})$ in the derived category of vector spaces. But $M$ is projective, hence this last complex is isomorphic to a shift of $\Hom(M,S_{\tau(x)})$. This space is zero by the hypothesis on $M$ as $\tau(x)$ is not frozen. The other case is analogous.
\end{proof}

\begin{lemma}\label{lemma:special resolutions in Rfilt}
For each non-frozen vertex $x$ of $\Z Q\fr{n}$ and $1 \leq l \leq n$, we have a minimal resolution of the simple $\Rc\filt{n}$-module $S_{\sigma^{-1}_l(x)}$ of the form:
\[
\dotsb \longrightarrow P_3 \longrightarrow P_2 \longrightarrow x^{\wedge} \longrightarrow \sigma^{-1}_l(x)^{\wedge} \longrightarrow S_{\sigma^{-1}_l(x)} \longrightarrow 0,
\]
where $P_t$ is a coproduct of free modules associated with \emph{frozen} vertices. Dually, we have a minimal coresolution of the simple $\Rc\filt{n}$-module $S_{\sigma_l(x)}$ of the form:
\[
0 \longrightarrow S_{\sigma_l(x)}  \longrightarrow \sigma_l(x)^{\vee} \longrightarrow x^{\vee} \longrightarrow I_2 \longrightarrow I_3 \longrightarrow \dotsb,
\]
where $I_t$ is a product of cofree modules associated with \emph{frozen} vertices.
\end{lemma}

\begin{proof}
Since Assumption \ref{assumption:Hom-finite and directed} holds for $\Rc\filt{n}$, this follows from a straightforward computation of $\Ext^p(S_{\sigma^{-1}_l(x)}, S_y)$ for $p \geq 1$ and a non-frozen vertex $y$ using the injective coresolution of $S_y$ from Lemma \ref{lemma:resolutions in Rfilt}. The second assertion follows dually.
\end{proof}

\begin{remark}
The resolution from Lemma \ref{lemma:special resolutions in Rfilt} is usually infinite. In particular, $\Rc\filt{n}$ does not have finite global dimension in general for $n \geq 2$.
\end{remark}

For an integer $1 \leq l \leq n$ and a non-frozen vertex $x$ in $\Z Q\fr{n}$ (that is, a vertex of $\Z Q$), consider the $\Rc\filt{n}$-modules
\begin{align*}
P^{\leq l}(x) &= \bigoplus_{y \in (\Z Q)_0} \Dc_Q(H(y),H(x)) \otimes (\sigma_1(y)^{\wedge} \oplus \dotsb \oplus \sigma_l(y)^{\wedge}),\\
I^{\geq l}(x) &= \prod_{y \in (\Z Q)_0} D\Dc_Q(H(x),H(y)) \otimes (\sigma^{-1}_l(y)^{\vee} \oplus \dotsb \oplus \sigma^{-1}_n(y)^{\vee}),
\end{align*}
where $\Dc_Q = \Dc^b(\modcat kQ)$ and $H$ is Happel's embedding. The first one is projective, while the second one is injective. Since the restriction of $\sigma_i(y)^{\wedge}$ (resp. $\sigma_i(y)^{\vee}$) to $\Mod\Sc\filt{n}$ is the free (resp. cofree) $\Sc\filt{n}$-module associated with $\sigma_i(y)$, the restrictions of $P^{\leq l}(x)$ and $I^{\geq l}(x)$ are again projective and injective, respectively. By abuse of notation, we also denote these restrictions by $P^{\leq l}(x)$ and $I^{\geq l}(x)$.

\begin{thm}\label{thm:resolutions of simple S-modules}
Suppose $Q$ is connected. Let $x$ be a non-frozen vertex of $\Z Q\fr{n}$ and $1 \leq l \leq n$. 
\begin{enumerate}[(a)]
\item If $Q$ is a Dynkin quiver, the simple $\Sc\filt{n}$-module $S_{\sigma_l^{-1}(x)}$ admits a minimal projective resolution of the form
\[
\dotsb \to P^{\leq l}(\Sigma^{-2} x) \to P^{\leq l}(\Sigma^{-1} x) \to P^{\leq l}(x) \to \sigma_l^{-1}(x)^\wedge \to S_{\sigma_l^{-1}(x)} \to 0
\]
and the simple $\Sc\filt{n}$-module $S_{\sigma_l(x)}$ admits a minimal injective coresolution of the form
\[
0 \to S_{\sigma_l(x)} \to \sigma_l(x)^\vee \to I^{\geq l}(x) \to I^{\geq l}(\Sigma x) \to I^{\geq l}(\Sigma^2 x) \to \dotsb.
\]
\item If $Q$ is not a Dynkin quiver, the simple $\Sc\filt{n}$-module $S_{\sigma_l^{-1}(x)}$ admits a minimal projective resolution of the form
\[
0 \to P^{\leq l}(x) \to \sigma_l^{-1}(x)^\wedge \to S_{\sigma_l^{-1}(x)} \to 0
\]
and the simple $\Sc\filt{n}$-module $S_{\sigma_l(x)}$ admits a minimal injective coresolution of the form
\[
0 \to S_{\sigma_l(x)} \to \sigma_l(x)^\vee \to I^{\geq l}(x)  \to 0.
\]
\end{enumerate}
\end{thm}

The proof follows the same arguments as in \cite[Section 3]{KellerScherotzke16}. First, we introduce auxiliary $\Rc\filt{n}$-modules. The quotient of $\Rc\filt{n}$ by the $k$-ideal generated by the identity morphisms of the frozen vertices is isomorphic to the mesh category $k(\Z Q)$. Therefore, by composition with the quotient functor, we may view the embedding $H$ of Theorem \ref{thm:Happel embedding} as a functor
\[
H: \Rc\filt{n} \longrightarrow \ind(\Dc_Q)
\]
vanishing on frozen vertices. We define the $\Rc\filt{n}$-modules
\[
x^{\wedge}_{\Dc} = \Dc_Q(H(?),H(x)) \quad \textrm{and} \quad x^{\vee}_{\Dc} = D\Dc_Q(H(x),H(?))
\]
for a non-frozen vertex $x \in (\Z Q\fr{n})_0$.

\begin{lemma} \label{lemma:resolutions of xD}
Suppose $Q$ is connected. Let $x$ be a non-frozen vertex of $\Z Q\fr{n}$.
\begin{enumerate}[(a)]
\item If $Q$ is a Dynkin quiver, we have a minimal resolution of $\Rc\filt{n}$-modules
\begin{equation*}
0 \longrightarrow (\Sigma^{-1} x)^\wedge \longrightarrow P^{\leq n}(x) \longrightarrow x^\wedge \longrightarrow x^\wedge_\Dc \longrightarrow 0
\end{equation*}
and a minimal coresolution
\begin{equation*}
0 \longrightarrow x^\vee_\Dc \longrightarrow x^\vee \longrightarrow I^{\geq 1}(x) \longrightarrow (\Sigma x)^\vee \longrightarrow 0.
\end{equation*}
\item If $Q$ is not a Dynkin quiver, we have a minimal resolution of $\Rc\filt{n}$-modules
\[
0 \longrightarrow P^{\leq n}(x) \longrightarrow x^\wedge \longrightarrow x^\wedge_\Dc \longrightarrow 0
\]
and a minimal coresolution
\[
0 \longrightarrow x^\vee_\Dc \longrightarrow x^\vee \longrightarrow I^{\geq 1}(x) \longrightarrow 0.
\]
\end{enumerate}
\end{lemma}

\begin{proof}
We sketch a proof for the resolutions of $x^{\wedge}_{\Dc}$. Since $\Rc\filt{n}$ satisfies Assumption \ref{assumption:Hom-finite and directed}, the proof boils down to the computation of $\Ext^n(x^{\wedge}_{\Dc}, S_u)$ for $n \geq 0$ and $u \in (\Z Q\fr{n})_0$. If $u = \sigma_l(y)$ is frozen, we may compute this extension group using the injective resolution of $S_{\sigma_l(y)}$ described in Lemma \ref{lemma:special resolutions in Rfilt}. Since $x^{\wedge}_{\Dc}$ vanishes on frozen vertices, one gets that $\Ext^n(x^{\wedge}_{\Dc}, S_{\sigma_l(y)}) = 0$ for $n \neq 1$ and
\[
\Ext^1(x^{\wedge}_{\Dc}, S_{\sigma_l(y)}) \cong \Hom(x^{\wedge}_{\Dc},y^{\vee}) = D\Dc_Q(H(y),H(x)).
\]
This gives the term $P^{\leq n}(x)$ in the resolutions above. The computation of $\Ext^n(x^{\wedge}_{\Dc}, S_u)$ for non-frozen vertices $u$ is identical to the one in the proof of \cite[Theorem 3.7]{KellerScherotzke16}.
\end{proof}

We can now sketch a proof of Theorem \ref{thm:resolutions of simple S-modules}.

\begin{proof}[Proof of Theorem \ref{thm:resolutions of simple S-modules}]
Identify $\Sc\filt{n}$ with $\Tc^n_{k\Sc_0}(\Sc)$ so that we can use the canonical functors from Section \ref{section:canonical functors}. If $1 \leq m \leq n$, notice that $i_{\sub}^{m,n}(\sigma_l^{\wedge}) = \sigma_l^{\wedge}$ and $i_{\quot}^{m,n}(\sigma_l^{\vee}) = \sigma_{l+n-m}^{\vee}$. These isomorphisms follow from the adjunctions coming from Lemma \ref{lemma:canonical functors as gluing data} and Section \ref{section:eight functors}. Consequently, we can recursively transport the (co)resolutions of the simple $\Sc\filt{m}$-modules to get almost all (co)resolutions of the simple $\Sc\filt{n}$-modules. The only ones missing are the projective resolution of $S_{\sigma_n^{-1}(x)}$ and the injective coresultion of $S_{\sigma_1(x)}$. One can compute them as in \cite{KellerScherotzke16}. We will do it explicitly in the Dynkin case.
    
Suppose $Q$ is Dynkin. Denote by $\res: \Mod\Rc\filt{n} \to \Mod\Sc\filt{n}$ the restriction functor. Since $H$ vanishes on frozen vertices, we have $\res(x^{\wedge}_{\Dc}) = 0$ and the resolution in Lemma \ref{lemma:resolutions of xD} induces an exact sequence
\begin{equation*}
0 \longrightarrow \res((\Sigma^{-1} x)^\wedge) \longrightarrow P^{\leq n}(x) \longrightarrow \res(x^\wedge) \longrightarrow 0
\end{equation*}
of $\Sc\filt{n}$-modules, and more generally, an exact sequence
\begin{equation*}
0 \longrightarrow \res((\Sigma^{-(p+1)} x)^\wedge) \longrightarrow P^{\leq n}(\Sigma^{-p} x) \longrightarrow \res((\Sigma^{-p} x)^\wedge) \longrightarrow 0.
\end{equation*}
Splicing these sequences with the restriction of the resolution (a) in Lemma \ref{lemma:resolutions in Rfilt}, we get the resolution
\[
\dotsb \to P^{\leq n}(\Sigma^{-2} x) \to P^{\leq n}(\Sigma^{-1} x) \to P^{\leq n}(x) \to \sigma_n^{-1}(x)^\wedge \to S_{\sigma_n^{-1}(x)} \to 0
\]
of $S_{\sigma_n^{-1}(x)}$, as desired. We can similarly find the injective coresolution of $S_{\sigma_1(x)}$.
\end{proof}

\begin{cor}\label{cor:Ext in S in terms of D}
Let $x,y \in (\Z Q\fr{n})_0$ be non-frozen vertices. For each $p \geq 1$, we have an isomorphism
\[
\Ext^p_{\Sc\filt{n}}(S_{\sigma_i(x)},S_{\sigma_j(y)}) \cong \begin{cases}
    0 &\textrm{if } i<j,\\
    \Dc_Q(H(x),\Sigma^pH(y)) &\textrm{if } i\geq j.
\end{cases}
\]
\end{cor}

\begin{proof}
    We may assume $Q$ is connected. Suppose first that $Q$ is Dynkin. We use the injective coresolution of $S_{\sigma_j(y)}$ from Theorem \ref{thm:resolutions of simple S-modules}. Since it is minimal, we get $\Ext^p(S_{\sigma_i(x)},S_{\sigma_j(y)}) \cong \Hom(S_{\sigma_i(x)}, I^{\geq j}(\Sigma^{p-1}y))$. Since
    \[
    I^{\geq j}(\Sigma^{p-1}y) = \prod_{z \in (\Z Q)_0}D\Dc_Q(\Sigma^{p-1}H(y),H(z)) \otimes (\sigma^{-1}_j(z)^{\vee} \oplus \dotsb \oplus \sigma^{-1}_n(z)^{\vee}),
    \]
    this morphism space vanishes if $i < j$. If $i \geq j$, it becomes $D\Dc_Q(\Sigma^{p-1}H(y),H(z))$ for $z$ satisfying $\sigma^{-1}_i(z) = \sigma_i(x)$, that is, for $z = \tau(x)$. Finally, by Serre duality in $\Dc_Q$, we have
    \[
    D\Dc_Q(\Sigma^{p-1}H(y),H(\tau(x))) \cong D\Dc_Q(\Sigma^pH(y),\Sigma\tau H(x)) \cong \Dc_Q(H(x),\Sigma^pH(y)).
    \]
    If $Q$ is not Dynkin, the same proof works for $p = 1$, and both sides vanish for $p \geq 2$.
\end{proof}

\subsection{Kan extensions}\label{section:Kan extensions} The restriction functor $\res: \Mod\Rc\filt{n} \to \Mod\Sc\filt{n}$ admits a left and a right adjoint. They are given by the left and right Kan extension functors $K_L$ and $K_R$ along the inclusion $\Sc\filt{n} \to \Rc\filt{n}$ (see \cite[Section X.3]{MacLane98}). Both of them are fully faithful by \cite[Corollary X.3.3]{MacLane98}, which implies that $\res$ is a localization of abelian categories (see \cite[Proposition 5, p. 374]{Gabriel62}).

\begin{lemma}\label{lemma:image of Kan extensions}
An $\Rc\filt{n}$-module $M$ belongs to the image of $K_L$, respectively $K_R$, if and only if, for every $N \in \ker(\res)$, we have
\[
\Hom(M,N) = 0 \quad \textrm{and} \quad \Ext^1(M,N) = 0,
\]
respectively,
\[
\Hom(N,M) = 0 \quad \textrm{and} \quad \Ext^1(N,M) = 0.
\]
If $Q$ is a Dynkin quiver, we may suppose $N = S_x$ for a non-frozen vertex $x \in \Z Q\fr{n}$.
\end{lemma}

\begin{proof}
The first statement is \cite[Lemme 1, p. 370]{Gabriel62} and \cite[Corollaire, p. 371]{Gabriel62}. The second can be proved as in \cite[Lemma 5.3]{KellerScherotzke16}.
\end{proof}

Denote by $\eta^R$ and $\varepsilon^R$ the unit and counit of the adjunction $\res \dashv K_R$. Similarly define $\eta^L$ and $\varepsilon^L$ for $K_L \dashv \res$. By \cite[Lemma 2.3]{KellerScherotzke14}, the following square is commutative:
\[\begin{tikzcd}
	{K_L\circ \res \circ K_R} & {K_L} \\
	{K_R} & {K_R\circ \res \circ K_L}
	\arrow["{K_L\varepsilon^R}", from=1-1, to=1-2]
	\arrow["{\varepsilon^LK_R}"', from=1-1, to=2-1]
	\arrow["{\eta^RK_L}", from=1-2, to=2-2]
	\arrow["{K_R\eta^L}", from=2-1, to=2-2]
\end{tikzcd}\]
Since $K_L$ and $K_R$ are fully faithful, $\eta^L$ and $\varepsilon^R$ are isomorphisms. In particular, the horizontal maps in the diagram above are isomorphisms and, if we invert them, we obtain a canonical morphism $\can: K_L \to K_R$.

\begin{lemma}\label{lemma:canonical morphism is iso for projective/injective}
    Suppose $Q$ is a Dynkin quiver. The canonical morphism $K_L(M) \to K_R(M)$ is an isomorphism if $M$ is a coproduct of free $\Sc\filt{n}$-modules or a product of cofree $\Sc\filt{n}$-modules.
\end{lemma}

\begin{proof}
We follow \cite[Lemma 5.4 (a)]{KellerScherotzke16}. Suppose that $M$ is a coproduct of free modules. We need to show that $\eta^R_{K_L(M)}$ is an isomorphism. Since $K_R$ is fully faithful, this is equivalent to showing that $K_L(M)$ is in the image of $K_R$. By Lemma \ref{lemma:image of Kan extensions}, it suffices to show that $\RHom(S_x,K_L(M)) = 0$ for any non-frozen vertex $x \in (\Z Q\fr{n})_0$. This follows from Corollary \ref{cor:RHom vanishes}. Indeed, knowing that $K_L$ is left adjoint to $\res$, one shows that $K_L(M)$ is a coproduct of free modules associated with frozen vertices. The proof when $M$ is a product of cofree modules is dual. 
\end{proof}

\begin{lemma}\label{lemma:crucial lemma to prove weakly Gorenstein}
Suppose $Q$ is a Dynkin quiver. Let $x$ be a non-frozen vertex of $\Z Q\fr{n}$.
\begin{enumerate}[(a)]
    \item Let
    \[\begin{tikzcd}
0 & {(\Sigma^{-1}x)^{\wedge}_{\Rc\filt{m}}} & {P_{\Rc\filt{m}}^{\leq m}(x)} & {x^{\wedge}_{\Rc\filt{m}}} & 0
\arrow[from=1-1, to=1-2]
\arrow[from=1-2, to=1-3]
\arrow[from=1-3, to=1-4]
\arrow[from=1-4, to=1-5]
    \end{tikzcd}\]
    be the resolution of $x^{\wedge}_{\Dc}$ in $\Mod\Rc\filt{m}$ for some $m \leq n$ from Lemma \ref{lemma:resolutions of xD}. Then its image under $\Hom_{\Sc\filt{n}}(i_{\sub}^{m,n}\circ \res(-), P)$ is acyclic for any finitely generated projective $\Sc\filt{n}$-module $P$.
    
    \item Let
    \[\begin{tikzcd}
0 & {x^{\vee}_{\Rc\filt{m}}} & {I_{\Rc\filt{m}}^{\geq 1}(x)} & {(\Sigma x)^{\vee}_{\Rc\filt{m}}} & 0
\arrow[from=1-1, to=1-2]
\arrow[from=1-2, to=1-3]
\arrow[from=1-3, to=1-4]
\arrow[from=1-4, to=1-5]
    \end{tikzcd}\]
    be the coresolution of $x^{\vee}_{\Dc}$ in $\Mod\Rc\filt{m}$ for some $m \leq n$ from Lemma \ref{lemma:resolutions of xD}. Then its image under $\Hom_{\Sc\filt{n}}(I, i_{\quot}^{m,n}\circ \res(-))$ is acyclic for any finitely cogenerated injective $\Sc\filt{n}$-module $I$.
\end{enumerate}   
\end{lemma}

\begin{proof}
For (a), the adjunctions $\res \dashv K_R$ and $i_{\sub}^{m,n} \dashv \sub^{n,m}$ imply that the functors
\[
\Hom_{\Sc\filt{n}}(i_{\sub}^{m,n}\circ \res(-), P) \quad \textrm{and} \quad \Hom_{\Rc\filt{m}}(-,K_R \circ \sub^{n,m}(P))
\]
are isomorphic. By Proposition \ref{prop:sub/quot preserve projective/injective} and Remark \ref{rem:sub/quot preserve fg proj/inj}, $\sub^{n,m}(P)$ is a coproduct of free $\Sc\filt{m}$-modules. Thus, $K_R(\sub^{n,m}(P))$ is isomorphic to $M = K_L(\sub^{n,m}(P))$ by Lemma \ref{lemma:canonical morphism is iso for projective/injective}. Due to the adjunction $K_L \dashv \res$, notice that $M$ is a coproduct of free $\Rc\filt{m}$-modules associated with frozen vertices. The statement we want to prove is equivalent to the vanishing of $\RHom_{\Rc\filt{m}}(x_{\Dc}^{\wedge},M)$ in the derived category of vector spaces. Since $x_{\Dc}^{\wedge}$ is finite-dimensional and concentrated on non-frozen vertices, it suffices to show that $\RHom(S_y,M)$ vanishes for any non-frozen vertex $y$. This follows from Corollary \ref{cor:RHom vanishes}. The proof of (b) is analogous.
\end{proof}

The image of the canonical map $\can: K_L \to K_R$ will be denoted by $K_{LR}$. Notice that $K_{LR}(M)$ is finite-dimensional for any finite-dimensional $\Sc\filt{n}$-module $M$. Indeed, since $K_L$ preserves finitely generated projectives and is right exact, $K_L(M)$ is pointwise finite-dimensional and right bounded. Dually, $K_R(M)$ is pointwise finite-dimensional and left bounded. But $K_{LR}(M)$ is a quotient of $K_L(M)$ and a submodule of $K_R(M)$, whence our claim. 

We also define $KK$ and $CK$ as the kernel and the cokernel of $\can$, respectively. One checks that $\res(\can)$ is an isomorphism, hence $KK(M)$ and $CK(M)$ are in the kernel of $\res$ for any $M \in \Mod\Sc\filt{n}$. Consequently, they can be viewed as modules over the quotient $\Rc\filt{n}/\ker(\res)$, which is isomorphic to the mesh category $k(\Z Q)$.

\begin{prop}\label{prop:KK and CK are finitely (co)generated proj/inj}
    If $M$ is a finite-dimensional $\Sc\filt{n}$-module, then $KK(M)$ is a finitely generated projective $k(\Z Q)$-module and $CK(M)$ is a finitely cogenerated injective $k(\Z Q)$-module.
\end{prop}

\begin{proof}
When $n = 1$, this is \cite[Theorem 4.8]{KellerScherotzke16}. Their proof works perfectly for $n \geq 1$ after a single adaptation, which we explain now. In the proof of \cite[Lemma 4.7]{KellerScherotzke16}, one needs to show that $\Ext^1(K_{LR}(M),N)$ is finite-dimensional if $N \in \ker(\res)$ is finite-dimensional. In loc. cit., it is proved that $\Ext^2(K_{LR}(M),N)$ vanishes and their argument works for general $n$. Hence, since $K_{LR}(M)$ and $N$ are both finite-dimensional, we may assume $N$ is simple. Since $\res(N) = 0$, we have $N \cong S_x$ for a non-frozen vertex $x \in (\Z Q)_0$. But $\Ext^1(K_{LR}(M),S_x)$ can be computed with the injective coresolution of $S_x$ given in Lemma \ref{lemma:resolutions in Rfilt} and its finite-dimensionality is evident.
\end{proof}

For the next lemma, we say that an $\Rc\filt{n}$-module $U$ is \emph{stable} (resp. \emph{costable}) if it does not admit any non-zero submodule (resp. quotient) concentrated in non-frozen vertices. From the adjunctions with the restriction functor, it follows, for example, that $K_R(M)$ is stable and $K_L(M)$ is costable for any $\Sc\filt{n}$-module $M$. Since $K_{LR}(M)$ is a submodule of $K_R(M)$ and a quotient of $K_L(M)$, $K_{LR}(M)$ is both stable and costable.

\begin{lemma}\label{lemma:Ext with stable and costable}
    If $U$ is a finite-dimensional $\Rc\filt{n}$-module which is both stable and costable, then 
    \[
    \dim\Ext^1(S_x,U) = \dim\Ext^1(U,S_{\tau(x)}) = (w_1 + \dotsb + w_n)(\sigma(x)) + (C_qv)(x)
    \]
    for any $x \in (\Z Q)_0$, where
    \[
    (C_qv)(x) = v(x) + v(\tau(x)) - \sum_{\substack{y \to x\\y \in (\Z Q)_0}}v(y)
    \]
    and $v,w_1,\dots,w_n$ are the dimension vectors of $U$ (as in Section \ref{section:link to quiver varieties}).
\end{lemma}

\begin{proof}
    The case $n = 1$ is \cite[Lemma 4.13]{KellerScherotzke16}, and the same proof works for $n \geq 2$.
\end{proof}

Let $M$ be a finite-dimensional $\Sc\filt{n}$-module. By Proposition \ref{prop:KK and CK are finitely (co)generated proj/inj}, there exist objects $\varphi_n(M), \psi_n(M) \in \Dc^b(\modcat kQ)$ such that
\[
CK(M) \cong D\Hom(\varphi_n(M),H(?)) \quad \textrm{and} \quad KK(M) \cong \Hom(H(?),\tau(\psi_n(M))),
\]
where we view $KK(M)$ and $CK(M)$ as $k(\Z Q)$-modules.

\begin{prop}\label{prop:KK and CK are compatible with tot}
    We have
    \[
    \varphi_1(\tot^n(M)) \cong \varphi_n(M) \quad \textrm{and} \quad \psi_1(\tot^n(M)) \cong \psi_n(M)
    \]
    for every finite-dimensional $\Sc\filt{n}$-module $M$.
\end{prop}

\begin{proof}
    We give a proof for the first isomorphism. The other can be proved similarly. Note that $\varphi_n(M)$ is a direct sum of objects of the form $H(x)$ with $x \in (\Z Q)_0$. The multiplicity of $H(x)$ in this direct sum is given by the dimension of $\Hom(S_x, CK(M))$. From the exact sequence
    \[
    0 \longrightarrow K_{LR}(M) \longrightarrow K_R(M) \longrightarrow CK(M) \longrightarrow 0
    \]
    and the fact that $\Hom(S_x,K_R(M)) = \Ext^1(S_x,K_R(M)) = 0$ by Lemma \ref{lemma:image of Kan extensions}, such multiplicity coincides with the dimension of $\Ext^1(S_x,K_{LR}(M))$, which is $(w_1 + \dotsb + w_n)(\sigma_n(x)) - (C_qv)(x)$ by Lemma \ref{lemma:Ext with stable and costable}, where $(v,w_1,\dots,w_n)$ is the dimension vector of $K_{LR}(M)$. Similarly, we prove that $\varphi_1(\tot^n(M))$ is a direct sum of objects of the form $H(x)$, where each term appears with multiplicity $w'(\sigma(x)) - (C_qv')(x)$, where $(v',w')$ denotes the dimension vector of the $\Rc$-module $K_{LR}(\tot^n(M))$. Therefore, the lemma will be proved once we show that $v = v'$ and $w' = w_1 + \dotsb + w_n$.

    For an $\Rc\filt{n}$-module $U$, denote by $\tot^n(U)$ the $\Rc$-module given on objects by
    \[
    \tot^n(U)(x) = U(x) \quad \textrm{and} \quad \tot^n(U)(\sigma(x)) = U(\sigma_1(x)) \oplus \dotsb \oplus U(\sigma_n(x))
    \]
    for $x \in (\Z Q)_0$, and with the evident action on morphisms. It can be promoted to a functor $\tot^n: \Mod\Rc\filt{n} \to \Mod\Rc$ compatible with the functor $\tot^n: \Mod\Sc\filt{n} \to \Mod\Sc$ and the restriction functors. To conclude the proof, it suffices to show that $\tot^n(K_{LR}(M)) \cong K_{LR}(\tot^n(M))$. Observe that the restrictions of these modules to $\Mod\Sc$ are both isomorphic to $\tot^n(M)$. Denote by $w: \Sc_0 \to \N$ the dimension vector of $\tot^n(M)$. Viewing $\tot^n(K_{LR}(M))$ and $K_{LR}(\tot^n(M))$ as points of an appropriate graded quiver variety, we deduce that they have the same image under the map
    \[
    \pi: \coprod\Mf(v,w) \longrightarrow \Mf_0(w),
    \]
    where the disjoint union varies over all dimension vectors $v: \Rc_0 \setminus \Sc_0 \to \N$. Now, notice that $\tot^n(K_{LR}(M))$ and $K_{LR}(\tot^n(M))$ are both stable and costable as $\tot^n$ clearly preserves these properties. Since $\pi$ is injective over the subset of points that are both stable and costable (see \cite[Section 2.6]{KellerScherotzke16} and the references therein), we obtain the desired isomorphism.
\end{proof}

\subsection{Gorenstein projective and injective modules}\label{section:Gorenstein modules} From now on, we assume that $Q$ is a Dynkin quiver. We will study the Gorenstein projective/injective $\Sc\filt{n}$-modules. To do so, it is helpful to prove that $\Sc\filt{n}$ is \emph{weakly Gorenstein of dimension $1$} (in the sense of \cite{KellerScherotzke16}), that is, we have
\[
\Ext_{\Sc\filt{n}}^p(M,P) = 0 = \Ext_{\Sc\filt{n}}^p(I,M)
\]
whenever $p \geq 2$, $P$ is a finitely generated projective module, $I$ is a finitely cogenerated injective module and $M$ is a finite-dimensional module. As in \cite[Lemma 5.7]{KellerScherotzke16}, we have the following slightly stronger property.

\begin{lemma}\label{lemma:weakly Gorenstein property}
For any finitely cogenerated injective $\Sc\filt{n}$-module $I$ and any pointwise finite-dimensional right bounded $\Sc\filt{n}$-module $M$, we have $\Ext^p(I,M) = 0$ for $p \geq 2$. Dually, for any finitely generated projective $\Sc\filt{n}$-module $P$ and any pointwise finite-dimensional left bounded $\Sc\filt{n}$-module $N$, we have $\Ext^p(N,P) = 0$ for $p \geq 2$.
\end{lemma}

\begin{proof}
We give a proof for the first statement. By the hypothesis on $M$, we can write $M$ as the inverse limit of a system of finite-dimensional modules. By the argument in \cite[Lemma 5.7]{KellerScherotzke16}, this reduces the proof to the case where $M$ is finite-dimensional, as long as we additionally show that $\Ext^p(I,M)$ is finite-dimensional for $p = 0,1$. Since $M$ is then finite length, we may further suppose that $M = S_{\sigma_l(x)}$ for some $x \in (\Z Q)_0$ and $1 \leq l \leq n$. To compute $\Ext^p(I,S_{\sigma_l(x)})$, we use the injective coresolution of $S_{\sigma_l(x)}$ from Theorem \ref{thm:resolutions of simple S-modules}. It is obtained by splicing the exact sequence of $\Rc\filt{(n-l+1)}$-modules
\[\begin{tikzcd}
	0 & {S_{\sigma_1(x)}} & {\sigma_1(x)^{\vee}} & {x^{\vee}} & 0
	\arrow[from=1-1, to=1-2]
	\arrow[from=1-2, to=1-3]
	\arrow[from=1-3, to=1-4]
	\arrow[from=1-4, to=1-5]
\end{tikzcd}\]
with the sequences
\[\begin{tikzcd}
	0 & {(\Sigma^{r-1}x)^{\vee}} & {I_{\Rc\filt{n-l+1}}^{\geq 1}(\Sigma^{r-1}x)} & {(\Sigma^rx)^{\vee}} & 0
	\arrow[from=1-1, to=1-2]
	\arrow[from=1-2, to=1-3]
	\arrow[from=1-3, to=1-4]
	\arrow[from=1-4, to=1-5]
\end{tikzcd}\]
for $r \geq 1$, and then applying the functors $\res$ and $i_{\quot}^{n-l+1,n}$. We conclude that $\Ext^p(I,S_{\sigma_l(x)})$ vanishes for $p \geq 2$ by Lemma \ref{lemma:crucial lemma to prove weakly Gorenstein}. We also deduce that $\Ext^p(I,S_{\sigma_l(x)})$ is finite-dimensional for $p = 0,1$ since $I$ and the modules in the coresolution of $S_{\sigma_l(x)}$ are finitely cogenerated.
\end{proof}

Following \cite{EnochsJenda95}, we say an $\Sc\filt{n}$-module $M$ is \emph{Gorenstein projective} if there is an exact complex
\[P = \begin{tikzcd}
	\dotsb & {P_1} & {P_0} & {P_{-1}} & \dotsb
	\arrow[from=1-1, to=1-2]
	\arrow["{d_1}", from=1-2, to=1-3]
	\arrow["{d_0}", from=1-3, to=1-4]
	\arrow[from=1-4, to=1-5]
\end{tikzcd}\]
of finitely generated projective modules such that $M \cong \coker d_1$ and $\Hom(P,P')$ is still exact for any finitely generated projective module $P'$. Dually, an $\Sc\filt{n}$-module $N$ is \emph{Gorenstein injective} if there is an exact complex
\[I = \begin{tikzcd}
	\dotsb & {I^{-1}} & {I^0} & {I^1} & \dotsb
	\arrow[from=1-1, to=1-2]
	\arrow["{d^{-1}}", from=1-2, to=1-3]
	\arrow["{d^0}", from=1-3, to=1-4]
	\arrow[from=1-4, to=1-5]
\end{tikzcd}\]
of finitely cogenerated injective modules such that $N \cong \ker d^0$ and $\Hom(I',I)$ is still exact for any finitely cogenerated injective module $I'$. Notice that in these definitions we are only considering finitely generated or cogenerated modules.

Let $\gpr(\Sc\filt{n})$ and $\gin(\Sc\filt{n})$ be the full subcategories of $\Mod\Sc\filt{n}$ consisting of Gorenstein projective and Gorenstein injective modules. The argument in \cite[Proposition 5.1]{AuslanderReiten91} shows that they are closed under extensions and direct summands. One can then easily check that they are Frobenius exact categories whose projective-injective objects are the finitely generated projective modules and the finitely cogenerated injective modules, respectively. We denote their stable categories (i.e., their quotients by the class of projective-injective objects) by $\gprs(\Sc\filt{n})$ and $\gins(\Sc\filt{n})$, which are canonically triangulated categories (see \cite{Happel88}).

For an $\Sc\filt{n}$-module $M$, define $\Omega M$ as the kernel of a fixed epimorphism $P_M \to M$ with $P_M$ projective. Dually, define $\Sigma M$ as the cokernel of a fixed monomorphism $M \to I_M$ with $I_M$ injective. If $M$ is finitely generated (resp. cogenerated), we suppose the same holds for $P_M$ (resp. $I_M$). For the simple modules, we take the canonical epimorphism $\sigma^{-1}_l(x)^{\wedge} \to S_{\sigma^{-1}_l(x)}$ and the canonical monomorphism $S_{\sigma_l(x)} \to \sigma_l(x)^{\vee}$, so that
\[
\Omega S_{\sigma^{-1}_l(x)} = i_{\sub}^{l,n} \circ \res(x^{\wedge}_{\Rc\filt{l}}) \quad \textrm{and} \quad \Sigma S_{\sigma_l(x)} = i_{\quot}^{n-l+1,n} \circ \res(x^{\vee}_{\Rc\filt{n-l+1}})
\]
by Theorem \ref{thm:resolutions of simple S-modules}. These constructions are not functorial on the level of $\Mod\Sc\filt{n}$, but they give rise to well-defined functors
\[
\Omega: \gprs(\Sc\filt{n}) \longrightarrow \gprs(\Sc\filt{n}) \quad \textrm{and} \quad \Sigma: \gins(\Sc\filt{n}) \longrightarrow \gins(\Sc\filt{n}),
\]
which are in fact equivalences. The second is the suspension functor of the canonical triangulated structure on $\gins(\Sc\filt{n})$ (whence the notation), while the first is a quasi-inverse for the suspension functor of $\gprs(\Sc\filt{n})$.

\begin{lemma}\label{lemma:syzygy of fd module is Gorenstein}
If $M$ is a finite-dimensional $\Sc\filt{n}$-module, then $\Omega M$ is Gorenstein projective and $\Sigma M$ is Gorenstein injective.
\end{lemma}

\begin{proof}
Since $\gpr(\Sc\filt{n})$ and $\gin(\Sc\filt{n})$ are extension-closed in $\Mod\Sc\filt{n}$, we may suppose $M = S_{\sigma^{-1}_l(x)}$ for some $x \in (\Z Q)_0$ and $1 \leq l \leq n$. For $p \in \Z$, consider the exact sequence of $\Sc\filt{l}$-modules
\[\begin{tikzcd}
0 & {\res(\Sigma^{p-1}x)^{\wedge}_{\Rc\filt{l}}} & {P_{\Sc\filt{l}}^{\leq l}(\Sigma^px)} & {\res(\Sigma^px)^{\wedge}_{\Rc\filt{l}}} & 0
\arrow[from=1-1, to=1-2]
\arrow[from=1-2, to=1-3]
\arrow[from=1-3, to=1-4]
\arrow[from=1-4, to=1-5]
\end{tikzcd}\]
obtained by restricting to $\Mod\Sc\filt{l}$ the resolution of $(\Sigma^px)^{\wedge}_{\Dc}$ in $\Mod\Rc\filt{l}$ from Lemma \ref{lemma:resolutions of xD}. Applying $i_{\sub}^{l,n}$ and splicing these sequences, we get an acyclic complex $P$ of finitely generated projective $\Sc\filt{n}$-modules. By Lemma \ref{lemma:crucial lemma to prove weakly Gorenstein}, $\Hom(P,P')$ is still acyclic. We conclude that $\Omega S_{\sigma^{-1}_l(x)} = i_{\sub}^{l,n}\circ \res(x^{\wedge}_{\Rc\filt{l}})$ is Gorenstein projective. Dually, one proves that $\Sigma S_{\sigma^{-1}_l(x)}$ is Gorenstein injective.
\end{proof}

\begin{lemma}\label{lemma:syzygy induces surjective map on Ext}
    If $M$ and $N$ are finite-dimensional $\Sc\filt{n}$-modules, then the map
    \[
    \Ext^1(M,N) \longrightarrow \Ext^1(\Omega M,\Omega N)
    \]
    induced by $\Omega$ is surjective. The dual result, where we replace $\Omega$ by $\Sigma$, also holds.
\end{lemma}

\begin{proof}
    We give a proof for the first statement. By the horseshoe lemma and the snake lemma, any extension of $M$ by $N$ gives rise to an extension of $\Omega M$ by $\Omega N$. This yields the map above. We can alternatively describe it using the derived category $\Dc(\Sc\filt{n})$ of $\Sc\filt{n}$-modules. View an extension $\xi \in \Ext^1(M,N)$ as a morphism $M \to N[1]$ in $\Dc(\Sc\filt{n})$, where $[1]$ denotes the shift functor of $\Dc(\Sc\filt{n})$. We also have canonical morphisms $\eta_M: M \to \Omega M[1]$ and $\eta_N: N \to \Omega N[1]$ induced by the exact sequences defining $\Omega M$ and $\Omega N$. We claim that the extension $\Omega\xi$ corresponds to the unique morphism $\Omega M \to \Omega N[1]$ such that the following square commutes:
    \[\begin{tikzcd}
	M & {\Omega M[1]} \\
	{N[1]} & {\Omega N[2]}
	\arrow["{\eta_M}", from=1-1, to=1-2]
	\arrow["\xi"', from=1-1, to=2-1]
	\arrow["{-\Omega\xi[1]}", from=1-2, to=2-2]
	\arrow["{\eta_N[1]}"', from=2-1, to=2-2]
    \end{tikzcd}\]
    The commutativity of the square and the reason for the negative sign in front of $\Omega\xi[1]$ follow from the $3 \times 3$ lemma for triangulated categories (see \cite[Proposition 1.1.11]{BeilinsonBernsteinDeligne81}). Applying $\Hom(-,\Omega N)$ to the exact sequence defining $M$, we obtain that composition with $\eta_M$ induces an isomorphism $\Ext^1(\Omega M,\Omega N) \to \Ext^2(M,\Omega N)$, whence the uniqueness in our claim. Therefore, to prove the lemma, it suffices to show that the map $\Ext^1(M,N) \to \Ext^2(M,\Omega N)$ given by composition with $\eta_N[1]$ is surjective. Since $\Ext^2(M,-)$ vanishes on finitely generated projective modules by Lemma \ref{lemma:weakly Gorenstein property}, this follows from the long exact sequence on extension groups obtained by applying $\Hom(M,-)$ to the exact sequence defining $\Omega N$.
\end{proof}

\subsection{The main equivalences} Let $\cev{\mathsf{A}}_n$ be a linearly oriented Dynkin diagram of type $\mathsf{A}_n$. The vertices are numbered from $1$ to $n$ in such a way that $1$ is the unique sink and $n$ is the unique source. We set $B_n = k\cev{\mathsf{A}}_n \otimes kQ$ and write $\Dc_{B_n}$ for the bounded derived category $\Dc^b(\modcat B_n)$ of finite-dimensional $B_n$-modules. In this section, we will adapt the proof of \cite[Theorem 5.18]{KellerScherotzke16} to prove the following result.

\begin{thm}\label{thm:equivalences for stable categories}
If $Q$ is a Dynkin quiver, then there are triangle equivalences
\[
F_n: \Dc_{B_n} \xlongrightarrow{\sim} \gins(\Sc\filt{n}) \quad \textrm{and} \quad G_n: \Dc_{B_n} \xlongrightarrow{\sim} \gprs(\Sc\filt{n}).
\]
\end{thm} 

As in the previous section, suppose $Q$ is Dynkin. For each $1 \leq i \leq n$, let $e_i \in k\cev{\mathsf{A}}$ be the stationary path corresponding to the vertex $i$. The endomorphism algebra of the projective $B_n$-module $(e_i \otimes 1)B_n$ is isomorphic to $kQ$. Thus, the derived tensor product with $(e_i \otimes 1)B_n$ induces a fully faithful triangle functor $\mu_i: \Dc_Q \to \Dc_{B_n}$. Since $(e_j \otimes 1)B_n(e_i \otimes 1)$ vanishes for $i < j$ and is otherwise isomorphic to $kQ$ as a left $kQ$-module, we have a natural isomorphism
\[
\Dc_{B_n}(\mu_i(M),\mu_j(N)) = \begin{cases}
    0 &\textrm{if }i < j,\\
    \Dc_Q(M,N) &\textrm{if }i \geq j,
\end{cases}
\]
for $M,N \in \Dc_Q$. We will construct $F_n$ and $G_n$ in such a way that $F_n(\mu_i(H(x))) = \Sigma S_{\sigma_i(x)}$ and $G_n(\mu_i(H(x))) = \Omega S_{\sigma_i(x)}$ for $x \in (\Z Q)_0$. In particular, we need the following lemma.

\begin{lemma}\label{lemma:morphisms in stable category}
For any two frozen vertices $\sigma_i(x)$ and $\sigma_j(y)$ in $\Z Q\fr{n}$ and any $p \in \Z$, there are isomorphisms between $\Dc_{B_n}(\mu_i(H(x)),\Sigma^p\mu_j(H(y)))$ and the morphism spaces
\[
\gins(\Sc\filt{n})(\Sigma S_{\sigma_i(x)},\Sigma^p\Sigma S_{\sigma_j(y)}) \quad \textrm{and} \quad \gprs(\Sc\filt{n})(\Omega S_{\sigma_i(x)},\Sigma^p\Omega S_{\sigma_j(y)}).
\]
\end{lemma}

\begin{proof}
We give a proof for the first isomorphism. Initially, let us suppose $p \geq 2$. Then
\begin{align*}
\gins(\Sc\filt{n})(\Sigma S_{\sigma_i(x)},\Sigma^p\Sigma S_{\sigma_j(y)}) &\cong \Ext_{\Sc\filt{n}}^p(\Sigma S_{\sigma_i(x)},\Sigma S_{\sigma_j(y)})\\
&\cong \Ext_{\Sc\filt{n}}^{p+1}(\Sigma S_{\sigma_i(x)},S_{\sigma_j(y)}),
\end{align*} 
where the second isomorphism follows from dimension shifting. Since $\Ext^q(-,S_{\sigma_j(y)})$ with $q \geq 2$ vanishes over finitely cogenerated injective modules by Lemma \ref{lemma:weakly Gorenstein property}, we can also apply dimension shifting in the first coordinate. Hence, the last space above is isomorphic to $\Ext^p(S_{\sigma_i(x)},S_{\sigma_j(y)})$, which in turn is isomorphic to $\Dc_{B_n}(\mu_i(H(x)),\Sigma^p\mu_j(H(y)))$ by Corollary \ref{cor:Ext in S in terms of D}.

Now, suppose $p$ is arbitrary. The short exact sequences obtained by applying $i_{\quot}^{n-j+1} \circ \res$ to the coresolution of $(\Sigma^qx)^{\vee}_{\Dc}$ in $\Mod\Rc\filt{(n-j+1)}$ from Lemma \ref{lemma:resolutions of xD} tell us that
\[
\Sigma^q\Sigma S_{\sigma_j(y)} \cong \Sigma S_{\sigma_j(\Sigma^qy)}
\]
in $\gins(\Sc\filt{n})$ for any $q \in \Z$. Consequently,
\begin{align*}
\gins(\Sc\filt{n})(\Sigma S_{\sigma_i(x)},\Sigma^p\Sigma S_{\sigma_j(y)}) &\cong \gins(\Sc\filt{n})(\Sigma S_{\sigma_i(x)},\Sigma^2\Sigma S_{\sigma_j(\Sigma^{p-2}y)})\\
&\cong \Dc_{B_n}(\mu_i(H(x)),\Sigma^2\mu_j(H(\Sigma^{p-2}y)))\\
&\cong \Dc_{B_n}(\mu_i(H(x)),\Sigma^p\mu_j(H(y))),
\end{align*}
where the second isomorphism follows from the previous paragraph.
\end{proof}

Consider the object
\[
T_n = \bigoplus_{v \in Q_0}\Sigma S_{\sigma_1((v,0))}\oplus\dotsb\oplus\Sigma S_{\sigma_n((v,0))} \in \gins(\Sc\filt{n}).
\]
Since $\mu_l(H((v,0)))$ gives all indecomposable projective $B_n$-modules as we vary $l$ and $v$, Lemma \ref{lemma:morphisms in stable category} implies that the endomorphism algebra of $T_n$ is isomorphic to $B_n$. We also deduce that
\[
\gins(\Sc\filt{n})(T_n,\Sigma^pT_n) = 0
\]
for any non-zero integer $p$. Since $B_n$ has finite global dimension and $\gins(\Sc\filt{n})$ is an algebraic triangulated category, there exists a fully faithful triangle functor $F_n: \Dc_{B_n} \to \gins(\Sc\filt{n})$ sending the regular $B_n$-module to $T_n$ (see \cite[Theorem 3.8]{Keller06}). A similar argument produces a fully faithful triangle functor $G_n: \Dc_{B_n} \to \gprs(\Sc\filt{n})$ sending $\mu_l(H((v,0)))$ to $\Omega S_{\sigma_l((v,0))}$ for all $1 \leq l \leq n$ and $v \in Q_0$.

\begin{lemma}\label{lemma:Fn and Gn have shifts of simple in the image}
For any $1 \leq l \leq n$ and any non-frozen vertex $x$ in $\Z Q\fr{n}$, the objects $\Sigma S_{\sigma_l(x)}$ and $\Omega S_{\sigma_l(x)}$ are in the essential image of $F_n$ and $G_n$, respectively.
\end{lemma}

\begin{proof}
By applying the restriction functor and the functor $i_{\quot}^{n-l+1,n}$ to the injective coresolution of $S_x$ in $\Mod\Rc\filt{(n-l+1)}$ from Lemma \ref{lemma:resolutions in Rfilt}, we get a short exact sequence of $\Sc\filt{n}$-modules
\[\begin{tikzcd}
	0 & {\Sigma S_{\sigma_l(x)}} & {\displaystyle\bigoplus_{l \leq i \leq n}\sigma^{-1}_i(x)^{\vee} \oplus \bigoplus_{x \to y}\Sigma S_{\sigma_l(y)}} & {\Sigma S_{\sigma_l(\tau^{-1}(x))}} & 0
	\arrow[from=1-1, to=1-2]
	\arrow[from=1-2, to=1-3]
	\arrow[from=1-3, to=1-4]
	\arrow[from=1-4, to=1-5]
\end{tikzcd}\]
where $y$ varies over all sucessors of $x$ in $\Z Q$. If we pass to the quotient $\gins(\Sc\filt{n})$, the cofree modules above vanish, and we get a distinguished triangle that resembles the Auslander--Reiten triangle in $\Dc_Q$ associated with the mesh starting at the vertex $x$. Since $\Sigma S_{\sigma_l((v,0))}$ belongs to the image of $F_n$ for all $v \in Q_0$ by construction, we can use these triangles to show that $F_n(\mu_l(H(x))) \cong \Sigma S_{\sigma_l(x)}$ for all $x \in (\Z Q)_0$ via a knitting argument. We can similarly prove the claim for the functor $G_n$.
\end{proof}

To conclude that $F_n$ and $G_n$ are equivalences, it remains to prove the following result.

\begin{lemma}
The functors $F_n$ and $G_n$ are essentially surjective.
\end{lemma}

\begin{proof}
We give a proof for the functor $F_n$. For $1 \leq l \leq n$, let $F^l_n: \Dc_Q \to \gins(\Sc\filt{n})$ be the composition $F_n \circ \mu_l$. We will prove that $F^l_n$ has a right adjoint $R^l_n$ or, equivalently, that the functor
\[
\gins(\Sc\filt{n})(F^l_n(-),M): \Dc_Q \longrightarrow \Mod k
\]
is representable for any $M \in \gins(\Sc\filt{n})$. Since the functor above is cohomological, \cite[Theorem 2.11]{BondalKapranov89} reduces our claim to showing that it is a functor of finite type, i.e., that it takes values on finite-dimensional vector spaces and vanishes for almost all indecomposable objects in $\Dc_Q$. Let $x \in (\Z Q)_0$. By Lemma \ref{lemma:Fn and Gn have shifts of simple in the image}, we have $F^l_n(H(x)) \cong \Sigma S_{\sigma_l(x)}$, which is isomorphic to $\Sigma^{-1}\Sigma S_{\sigma_l(\Sigma x)}$ by the argument in the proof of Lemma \ref{lemma:morphisms in stable category}. Hence,
\begin{align*}
    \gins(\Sc\filt{n})(F^l_n(H(x)),M) &\cong \gins(\Sc\filt{n})(\Sigma^{-1}\Sigma S_{\sigma_l(\Sigma x)},M)\\
    &\cong \gins(\Sc\filt{n})(\Sigma S_{\sigma_l(\Sigma x)},\Sigma M)\\
    &\cong \Ext^1_{\Sc\filt{n}}(\Sigma S_{\sigma_l(\Sigma x)},M). 
\end{align*}
Using that $\Sigma S_{\sigma_l(\Sigma x)}$ is the cokernel of the map $S_{\sigma_l(\Sigma x)} \to \sigma_l(\Sigma x)^{\vee}$ and that $\Ext^1(\sigma_l(\Sigma x)^{\vee},M)$ vanishes (as $M$ is Gorenstein injective), we deduce that the last space above is the cokernel of the map
\[
\Hom(\sigma_l(\Sigma x)^{\vee},M) \longrightarrow \Hom(S_{\sigma_l(\Sigma x)},M).
\]
But the socle of $M$ is finite-dimensional as $M$ is finitely cogenerated, hence the target of the map above is always finite-dimensional and vanishes for almost all $x$, as desired.

Let $M \in \gins(\Sc\filt{n})$ and let us show it is in the essential image of $F_n$. We will define recursively a sequence of objects $M_1, M_2, \dots, M_{n+1}$ as follows. Put $M_1 = M$. For $1 \leq l \leq n$, define $M_{l+1}$ as the cone of the counit map $F^l_nR^l_n(M_l) \to M_l$ from the adjunction $F^l_n \dashv R^l_n$, so that we have a distinguished triangle
\[\begin{tikzcd}
	{F^l_nR^l_n(M_l)} & M_l & {M_{l+1}} & {\Sigma F^l_nR^l_n(M)}
	\arrow[from=1-1, to=1-2]
	\arrow[from=1-2, to=1-3]
	\arrow[from=1-3, to=1-4]
\end{tikzcd}\]
in $\gins(\Sc\filt{n})$. We claim that $M_{l+1}$ is right orthogonal to the image of the functors $F^1_n$, $F^2_n$, $\dots$, $F^l_n$, that is, $\gins(\Sc\filt{n})(F^{l'}_n(X),M_{l+1})$ vanishes for all $1 \leq l' \leq l$ and $X \in \Dc_Q$. To see this, observe first that the image of $F^l_n$ is right orthogonal to the images of the functors $F^1_n, \dots, F^{l-1}_n$. This follows from Lemmas \ref{lemma:morphisms in stable category} and \ref{lemma:Fn and Gn have shifts of simple in the image}. Since the right orthogonality property is closed under cones, by induction, it suffices to verify that $M_{l+1}$ is right orthogonal to the image of $F^l_n$. This is true because $F^l_n$ is fully faithful, hence the map
\[
\gins(\Sc\filt{n})(N,F^l_nR^l_n(M_l)) \longrightarrow \gins(\Sc\filt{n})(N,M_l)
\]
is an isomorphism for every object $N$ in the image of $F^l_n$.

To conclude the proof, it suffices to show that $M_{n+1} = 0$ in $\gins(\Sc\filt{n})$. Indeed, knowing that $F_n$ is a full triangle functor, we can use the triangles above to deduce that all $M_l$, including $M_1 = M$, are in the image of $F_n$. By the previous paragraph, $M_{n+1}$ is right orthogonal to the image of the functors $F^l_n$ for $1 \leq l \leq n$. Let us show that this right orthogonal subcategory is zero. Take $M$ in this subcategory. Since $M$ is finitely cogenerated, it has a finite-dimensional socle. We prove that $M$ is an injective $\Sc\filt{n}$-module by induction on the dimension of its socle. If this dimension is zero, then $M$ is the zero module, and we are done. If not, then it contains a simple submodule of the form $S_{\sigma_l(x)}$. By the hypothesis on $M$, the space
\[
\gins(\Sc\filt{n})(F^l_n(H(\Sigma^{-1}x)),M) \cong \coker\left(\Hom(\sigma_l(x)^{\vee},M) \longrightarrow \Hom(S_{\sigma_l(x)},M)\right)
\]
is zero. In particular, the inclusion $S_{\sigma_l(x)} \to M$ extends to a morphism $\sigma_l(x)^{\vee} \to M$, which must be injective as it is injective on the socles. But $\sigma_l(x)^{\vee}$ is an injective module, so it is a direct summand of $M$. The complement $M/\sigma_l(x)^{\vee}$ is isomorphic to $M$ in $\gins(\Sc\filt{n})$ (so it is also in the right orthogonal subcategory) and its socle is strictly smaller than that of $M$. By the induction hypothesis, this complement is injective and so $M$ is also injective. This finishes the proof.
\end{proof}

By the proof of Theorem \ref{thm:equivalences for stable categories}, there are simple $\Sc\filt{n}$-modules $S_1,\dots,S_r$ such that $\gprs(\Sc\filt{n})$ is the thick triangulated subcategory generated by $\Omega S_1,\dots, \Omega S_r$. This observation gives the following recursive description of $\gprs(\Sc\filt{n})$. Denote by $E_1$ the strictly full subcategory of $\gprs(\Sc\filt{n})$ whose objects are isomorphic to a direct sum of shifts of the objects $\Omega S_1,\dots, \Omega S_r$. For $l \geq 2$, define $E_l$ recursively as the strictly full subcategory of objects isomorphic to a direct summand of an object $X$ that fits into a distinguished triangle
\begin{equation}\label{eq:triangle from generator}
    A \longrightarrow X \longrightarrow B \longrightarrow \Sigma A
\end{equation}
where $A \in E_1$ and $B \in E_{l-1}$. Then $\gprs(\Sc\filt{n})$ is the union of the $E_l$ for $l \geq 1$. Dually, $\gins(\Sc\filt{n})$ is the thick triangulated subcategory generated by $\Sigma S_1,\dots,\Sigma S_r$ and can be similarly described. We obtain two important consequences.

\begin{cor}\label{cor:canonical functors preserve Gorenstein modules}
    For $m \geq n \geq 1$, the functors $i^{n,m}_{\sub}$ and $\quot^{m,n}$ preserve Gorenstein projective modules, while $i^{n,m}_{\quot}$ and $\sub^{m,n}$ preserve Gorenstein injective modules.
\end{cor}

\begin{proof}
    We prove the first part. For $m,n \geq 1$, let $F: \Mod\Sc\filt{n} \to \Mod\Sc\filt{m}$ be an exact functor which preserves finitely generated projective modules and finite-dimensional modules. For example, these properties are satisfied if $F = i_{\sub}^{n,m}$ or $F = \quot^{n,m}$. We show that $F$ preserves Gorenstein projective modules.

    Let $M \in \gpr(\Sc\filt{n})$. Viewed as an object of the stable category, it belongs to $E_l$ for some $l \geq 1$. We prove that $F(M)$ is Gorenstein projective by induction on $l$. Since $F$ is exact and preserves finitely generated projectives, it sends $\Sigma^p (\Omega S)$ for a simple module $S$ and $p \in \Z$ to the direct sum of $\Sigma^p(\Omega F(S))$ with a finitely generated projective. But $F(S)$ is finite-dimensional, hence this direct sum is Gorenstein projective by Lemma \ref{lemma:syzygy of fd module is Gorenstein}. In particular, if $l = 1$, this implies that $F(M)$ is Gorenstein projective. If $l \geq 2$, then $M$ is a direct summand of an object $X$ that sits in a triangle as in (\ref{eq:triangle from generator}). Consequently, by adding projective summands if necessary, we find an exact sequence
    \[
    0 \longrightarrow A \longrightarrow M \oplus N \longrightarrow B \longrightarrow 0
    \]
    in $\gpr(\Sc\filt{n})$ where $A \in E_1$ and $B \in E_{l-1}$. By the induction hypothesis, $F(A)$ and $F(B)$ are Gorenstein projective, thus so is $F(M)$ since $F$ is exact and $\gpr(\Sc\filt{m})$ is closed under extensions and direct summands.
\end{proof}

\begin{cor}\label{cor:syzygies give all Gorenstein modules}
    Any Gorenstein projective $\Sc\filt{n}$-module is a direct summand of $\Omega X \oplus P$ for some finite-dimensional module $X$ and some finitely generated projective module $P$. Dually, any Gorenstein injective $\Sc\filt{n}$-module is a direct summand of $\Sigma Y \oplus I$ for some finite-dimensional module $Y$ and some finitely cogenerated injective module $I$.
\end{cor}

\begin{proof}
    We prove the first statement. Let $M \in \gprs(\Sc\filt{n})$. It belongs to $E_l$ for some $l \geq 1$. We prove by induction on $l$ that $M$ is a direct summand of $\Omega M'$ (in the stable category) for some finite-dimensional module $M'$. By the same argument as in the proof of Lemma \ref{lemma:morphisms in stable category}, any shift of $\Omega S$ with $S$ simple is again of the form $\Omega S'$ for a simple module $S'$. Hence, our claim is true for $l = 1$. If $l \geq 2$, then $M$ is a direct summand of an object $X$ that fits into a distinguished triangle
    \[
    A \longrightarrow X \longrightarrow B \longrightarrow \Sigma A
    \]
    where $A \in E_1$ and $B \in E_{l-1}$. By the induction hypothesis, we may add direct summands and assume that $A \cong \Omega A'$ and $B \cong \Omega B'$ where $A'$ and $B'$ are finite-dimensional. By Lemma \ref{lemma:syzygy induces surjective map on Ext}, there is an extension $X'$ of $B'$ by $A'$ such that $X \cong \Omega X'$. Since $M$ is a direct summand of $X$, this concludes the proof.
\end{proof}

\begin{remark}\label{rem:extension of syzygies is syzygy}
    The same argument shows that if $M \in \gprs(\Sc\filt{n})$ is an iterated extension of objects of the form $\Omega X$ with $X$ finite-dimensional, then $M$ is also of this form.
\end{remark}

\subsection{The stratification functors}

We continue to assume that $Q$ is a Dynkin quiver. Let $\modcat\Sc\filt{n}$ be the category of finite-dimensional $\Sc\filt{n}$-modules. By Lemma \ref{lemma:syzygy of fd module is Gorenstein}, we have the (co)syzygy functors
\[
\Omega: \modcat\Sc\filt{n} \longrightarrow \gprs(\Sc\filt{n}) \quad \textrm{and} \quad \Sigma: \modcat\Sc\filt{n} \longrightarrow \gins(\Sc\filt{n}).
\]
We define the \emph{stratification functors}
\[
\Phi^n, \Psi^n: \modcat\Sc\filt{n} \longrightarrow \Dc_{B_n} = \Dc^b(\modcat k\cev{\mathsf{A}}_n \otimes kQ)
\]
as the compositions $G_n^{-1} \circ \Omega$ and $F_n^{-1} \circ \Sigma$, respectively, where $F_n^{-1}$ and $G_n^{-1}$ are quasi-inverses for the equivalences from Theorem \ref{thm:equivalences for stable categories}. When $n=1$, we recover the stratification functor of \cite{KellerScherotzke16}. We remark that $\Phi^n$ and $\Psi^n$ are $\delta$-functors in the sense of \cite{Keller91}. This essentially means that they send short exact sequences to distinguished triangles in a functorial way. 

\begin{remark}
In \cite{KellerScherotzke16}, it is implicitly shown that $\Phi^n$ is naturally isomorphic to $\Psi^n$ when $n = 1$. This follows from Proposition 2.13 and Lemma 4.14 (and its dual version) in loc. cit. We do not know if this property holds for $n \geq 2$.
\end{remark}

Our next goal is to show some compatibility relations between the functors $\Phi^n$ and $\Psi^n$ and the canonical functors from Section \ref{section:canonical functors}.

For $1 \leq i_1 < i_2 < \dotsb < i_m \leq n$, there is a canonical truncation map $k\cev{A_n} \to k\cev{A_m}$ that sends the unique path between $i_r$ and $i_s$ to the unique path between $r$ and $s$, and the other paths to zero. Tensoring with $kQ$, we get an algebra morphism $t^{i_1,\dots,i_m}_n: B_n \to B_m$. By abuse of notation, we also denote by $t^{i_1,\dots,i_m}_n$ the derived tensor product functor $- \otimes^L_{B_n} B_m: \Dc_{B_n} \to \Dc_{B_m}$, where we view $B_m$ as a left $B_n$-module via this truncation map.

\begin{thm}\label{thm:compatibility without tot}
For $1 \leq m \leq n$, we have the following isomorphisms of functors:
    \begin{align*}
    t^{1,2,\dots,m}_n \circ \Psi^n &\cong \Psi^m \circ \sub^{n,m},\\
    t^{n-m+1,\dots,n-1,n}_n \circ \Phi^n &\cong \Phi^m \circ \quot^{n,m}.
    \end{align*}
\end{thm}

\begin{proof}
    We prove the first isomorphism. The other one can be proved dually. By Corollary \ref{cor:canonical functors preserve Gorenstein modules} and Proposition \ref{prop:sub/quot preserve projective/injective}, the functor $\sub^{n,m}$ descends to the stable category of Gorenstein injective modules and the left square in the diagram
    \[\begin{tikzcd}
	{\modcat\Sc\filt{n}} & {\gins(\Sc\filt{n})} & {\Dc_{B_n}} \\
	{\modcat\Sc\filt{m}} & {\gins(\Sc\filt{m})} & {\Dc_{B_m}}
	\arrow["\Sigma", from=1-1, to=1-2]
	\arrow["{\sub^{n,m}}"', from=1-1, to=2-1]
	\arrow["{\sub^{n,m}}", from=1-2, to=2-2]
	\arrow["{F_n}"', from=1-3, to=1-2]
	\arrow["{t^{1,2,\dots,m}_n}", from=1-3, to=2-3]
	\arrow["\Sigma"', from=2-1, to=2-2]
	\arrow["{F_m}", from=2-3, to=2-2]
    \end{tikzcd}\]
    commutes up to natural isomorphism. We prove the same is true for the square on the right. This will conclude the proof.

    Let us first show that both compositions in the right square send the regular $B_n$-module to the object $T_m \in \gins(\Sc\filt{m})$ from the proof of Theorem \ref{thm:equivalences for stable categories}. On one hand, we have
    \[
    \sub^{n,m}(\Sigma S_{\sigma_i(x)}) \cong \Sigma \, \sub^{n,m}(S_{\sigma_i(x)}) = \begin{cases}
        \Sigma S_{\sigma_i(x)} &\textrm{if } 1 \leq i \leq m,\\
        0 &\textrm{otherwise}.
    \end{cases}
    \]
    We deduce that $\sub^{m,n}(F_n(B_n)) = \sub^{m,n}(T_n) \cong T_m$. On the other hand, $t^{1,2,\dots,m}_n(B_n) = B_m$, which is sent to $T_m$ by $F_m$.

    We now verify that both actions of $B_n$ on $T_m$ defined by the functors above are compatible. In other words, we check that the diagram
    \[\begin{tikzcd}
	{B_n = \Dc_{B_n}(B_n,B_n)} & {\gins(\Sc\filt{n})(T_n,T_n)} \\
	{B_m = \Dc_{B_m}(B_m,B_m)} & {\gins(\Sc\filt{m})(T_m,T_m)}
	\arrow["\sim", from=1-1, to=1-2]
	\arrow["{t^{1,2,\dots,m}_n}"', from=1-1, to=2-1]
	\arrow["{\sub^{n,m}}", from=1-2, to=2-2]
	\arrow["\sim", from=2-1, to=2-2]
    \end{tikzcd}\]
    is commutative. It suffices to show that, for any $x,y \in (\Z Q)_0$ and $1 \leq i,j \leq m$ with $i \geq j$, the morphism
    \[
    \gins(\Sc\filt{n})(\Sigma S_{\sigma_i(x)}, \Sigma S_{\sigma_j(y)}) \longrightarrow \gins(\Sc\filt{m})(\Sigma S_{\sigma_i(x)}, \Sigma S_{\sigma_j(y)})
    \]
    yielded by $\sub^{n,m}$ coincides with the identity map $\Dc_Q(H(x),H(y)) \to \Dc_Q(H(x),H(y))$ under the identifications of Lemma \ref{lemma:morphisms in stable category}. Since $\sub^{n,m}$ is exact and preserves injective modules, we can use the isomorphisms in the proof of this lemma to identify the map above with the map induced on extension groups:
    \[
    \Ext^2_{\Sc\filt{n}}(S_{\sigma_i(x)},S_{\sigma_j(\Sigma^{-2}y)}) \longrightarrow \Ext^2_{\Sc\filt{m}}(S_{\sigma_i(x)},S_{\sigma_j(\Sigma^{-2}y)}).
    \]
    We use the injective coresolution of $S_{\sigma_j(\Sigma^{-2} y)}$ from Theorem \ref{thm:resolutions of simple S-modules} to compute the groups above. Observe that $\sub^{n,m}$ sends the injective module $I^{\geq j}_{\Sc\filt{n}}(z)$ to $I^{\geq j}_{\Sc\filt{m}}(z)$ for any $z \in (\Z Q)_0$. Hence, the map above becomes the map
    \[
    \Hom_{\Sc\filt{n}}(S_{\sigma_i(x)},I^{\geq j}_{\Sc\filt{n}}(\Sigma^{-1}y)) \longrightarrow \Hom_{\Sc\filt{m}}(S_{\sigma_i(x)},I^{\geq j}_{\Sc\filt{m}}(\Sigma^{-1}y))
    \]
    induced by $\sub^{n,m}$. Computing the spaces above as in Corollary \ref{cor:Ext in S in terms of D} and knowing that $\sub^{n,m}$ acts as the identity on the $i$-th component of a split filtration (as $i \leq m$), the map above identifies with the identity on $\Dc_Q(H(x),H(y))$, as desired.

    We are now ready to finish the proof using some properties of algebraic triangulated categories (see \cite{Keller06}). The categories $\gins(\Sc\filt{n})$ and $\gins(\Sc\filt{m})$ can each be viewed as a full triangulated subcategory of the derived category of some dg category. In this way, since $\sub^{n,m}$ is an exact functor and preserves the projective-injective objects of $\gin(\Sc\filt{n})$, it can be identified with the derived tensor product functor associated with a bimodule on the level of these derived categories. By construction, $F_n$, $F_m$ and $t^{1,2,\dots,m}_n$ are also of this form. Consequently, if $\Ac$ is the dg category above whose derived category $\Dc(\Ac)$ contains $\gins(\Sc\filt{m})$, then there are $B_n$-$\Ac$-bimodules $X$ and $Y$ such that $\sub^{n,m} \circ F_n$ and $F_m \circ t^{1,2,\dots,m}_n$ identify with the functors $- \otimes^{L}_{B_n} X$ and $- \otimes^{L}_{B_n} Y$ from $\Dc_{B_n}$ to $\Dc(\Ac)$. By the previous paragraphs, the restrictions of $X$ and $Y$ to $\Ac$ are both isomorphic to $T_m$ and the homotopy actions $B_n \to \Hom_{\Dc(\Ac)}(T_m,T_m)$ induced by the left action of $B_m$ on $X$ and $Y$ agree. Since $T_m$ has no negative self-extensions, this information allows one to construct an isomorphism between $X$ and $Y$ in the derived category of $B_n$-$\Ac$-bimodules by \cite[Theorem 2.1]{Keller00} (see also \cite[Lemma 2.6]{HuaKeller24}). This yields the desired natural isomorphism from $\sub^{n,m} \circ F_n$ to $F_m \circ t^{1,2,\dots,m}_n$. 
\end{proof}

For $1 \leq i \leq n$, let $\mu_i: \Dc_Q \to \Dc_{B_n}$ be the functor from the previous section. We denote its left and right adjoints by $\mu_i^L$ and $\mu_i^R$.

\begin{thm}\label{thm:compatibility with tot}
    For $1 \leq m \leq n$, we have isomorphisms
    \begin{align*}
    (\mu_m^R \circ \Phi^n)(M) &\cong (\Phi^1 \circ \tot^m \circ \sub^{n,m})(M)\\
    (\mu_{n-m+1}^L \circ \Psi^n)(M) &\cong (\Psi^1 \circ \tot^m \circ \quot^{n,m})(M)
    \end{align*}
    for any $M \in \modcat\Sc\filt{n}$.
\end{thm}

\begin{proof}
    We only give a proof for the first isomorphism. By \cite[Proposition 2.13]{KellerScherotzke16}, $\Phi^1$ coincides with the functor $\varphi_1$ from Section \ref{section:Kan extensions}.  By Proposition \ref{prop:KK and CK are compatible with tot}, the proof reduces to finding an isomorphism of $k(\Z Q)$-modules
    \[
    CK(\sub^{n,m}(M)) \cong D\Dc_{Q}(\mu_m^R\Phi^n(M),H(?))
    \]
    for $M \in \modcat\Sc\filt{n}$. We do it by adapting the proof of \cite[Proposition 2.13]{KellerScherotzke16}.

    Let $P \to M$ be an epimorphism where $P$ is a finitely generated projective module. It is still surjective after applying the exact functor $\sub^{n,m}$ and induces the following commutative diagram:
    \begin{equation}\label{eq:diagram in the proof of compatibility of Phis}
    \begin{tikzcd}
	{K_L(\sub^{n,m}(P))} & {K_R(\sub^{n,m}(P))} \\
	{K_L(\sub^{n,m}(M))} & {K_R(\sub^{n,m}(M))}
	\arrow[from=1-1, to=1-2]
	\arrow[from=1-1, to=2-1]
	\arrow[from=1-2, to=2-2]
	\arrow[from=2-1, to=2-2]
    \end{tikzcd}
    \end{equation}
    The map on the left is again surjective as $K_L$ is a left adjoint. By Proposition \ref{prop:sub/quot preserve projective/injective} and Remark \ref{rem:sub/quot preserve fg proj/inj}, $\sub^{n,m}(P)$ is a coproduct of free modules, so Lemma \ref{lemma:canonical morphism is iso for projective/injective} implies that the top map in the diagram is an isomorphism. We conclude that the cokernel of the bottom map, i.e. $CK(\sub^{n,m}(M))$, is isomorphic to the cokernel of the map on the right. 
    
    Let us now compute this last cokernel on a non-frozen vertex $x \in \Z Q$. First, we have
    \begin{align*}
    K_R(\sub^{n,m}(M))(x) \cong \Hom(x^{\wedge},K_R(\sub^{n,m}(M))) &\cong \Hom(i_{\sub}^{m,n} \circ \res(x^{\wedge}), M)\\
    &= \Hom(\Omega S_{\sigma^{-1}_m(x)}, M)
    \end{align*}
    and similarly for $P$ instead of $M$. Thus, we have the commutative diagram
    \[\begin{tikzcd}
	{\Hom(\sigma_m^{-1}(x)^{\wedge},P)} & {\Hom(\Omega S_{\sigma_m^{-1}(x)},P)} & {\Ext^1(S_{\sigma_m^{-1}(x)}, P)} & 0 \\
	{\Hom(\sigma_m^{-1}(x)^{\wedge},M)} & {\Hom(\Omega S_{\sigma_m^{-1}(x)},M)} & {\Ext^1(S_{\sigma_m^{-1}(x)}, M)} & 0
	\arrow[from=1-1, to=1-2]
	\arrow[from=1-1, to=2-1]
	\arrow[from=1-2, to=1-3]
	\arrow[from=1-2, to=2-2]
	\arrow[from=1-3, to=1-4]
	\arrow[from=1-3, to=2-3]
	\arrow[from=2-1, to=2-2]
	\arrow[from=2-2, to=2-3]
	\arrow[from=2-3, to=2-4]
    \end{tikzcd}\]
    where the rows are exact and where the second vertical map identifies with the map on the right in (\ref{eq:diagram in the proof of compatibility of Phis}) evaluated at $x$. But the first vertical map is surjective as $\sigma_m^{-1}(x)^{\wedge}$ is projective, hence the cokernel of the second vertical map is isomorphic to that of the third one. By Lemma \ref{lemma:weakly Gorenstein property}, $\Ext^2(S_{\sigma^{-1}_m(x)},P)$ is zero, so the cokernel of the third vertical map above is isomorphic to $\Ext^2(S_{\sigma^{-1}_m}(x),\Omega M)$. We then have the identifications:
    \[
    \Ext^2(S_{\sigma^{-1}_m}(x),\Omega M) \cong \Ext^1(\Omega S_{\sigma^{-1}_m}(x),\Omega M) \cong \gprs(\Sc\filt{n})(\Omega S_{\sigma^{-1}_m}(x),\Sigma\Omega M).
    \]
    Now, if $G_n$ denotes the equivalence from Theorem \ref{thm:equivalences for stable categories}, we have
    \[
    G_n(\mu_m(\tau^{-1}H(x))) = G_n(\mu_m(H(\tau^{-1}(x)))) = \Omega S_{\sigma_m(\tau^{-1}(x))} = \Omega S_{\sigma_m^{-1}(x)}.
    \]
    Therefore, since $\Phi^{n} = G_n^{-1} \circ \Omega$, we have
    \begin{align*}
        \gprs(\Sc\filt{n})(\Omega S_{\sigma^{-1}_m}(x),\Sigma\Omega M) &\cong \Dc_{B_n}(\mu_m(\tau^{-1}H(x)),\Sigma\Phi^n(M))\\
        &\cong \Dc_Q(\tau^{-1}H(x), \Sigma\mu_m^R\Phi^n(M))\\
        &\cong \Dc_Q(H(x), \Sigma\tau\mu_m^R\Phi^n(M))\\
        &\cong D\Dc_Q(\mu_m^R\Phi^n(M),H(x)),
    \end{align*}
    where the last isomorphism is Serre duality in $\Dc_Q$. This concludes the proof.
\end{proof}

Let $w_1,\dots,w_n: \Sc_0 \to \N$ be dimension vectors. We define an equivalence relation on the $n$-fold affine graded tensor product variety $\Tf_0(w_1,\dots,w_n)$ as follows. Identifying two points $M, N \in \Tf_0(w_1,\dots,w_n)$ as finite-dimensional $\Sc\filt{n}$-modules, we say that $M$ and $N$ are equivalent if and only if $\Phi^n(M) \cong \Phi^n(N)$ in $\Dc_{B_n}$. When $n = 1$, the induced partition on $\Mf_0(w_1)$ coincides with Nakajima's stratification \cite{Nakajima01findim,Nakajima11} by \cite[Theorem 2.7]{KellerScherotzke16}.

In general, this partition contains information about the stratum of the subquotients of the filtration of $M$. To state this more precisely, note that the canonical functors $\sub^{n,m}$ and $\quot^{n,m}$ can be viewed as morphisms of varieties from $\Tf_0(w_1,\dots,w_n)$ to $\Tf_0(w_1,\dots,w_m)$ and $\Tf_0(w_{n-m+1}, \dots,w_n)$, respectively. For $0 \leq i < j \leq n$, the composition
\[
\subquot^n_{i,j}: \Tf_0(w_1,\dots,w_n) \xlongrightarrow{\sub^{n-i,j-i} \ \circ \ \quot^{n,n-i}} \Tf_0(w_{i+1},\dots,w_j) \hookrightarrow \Mf_0(w_{i+1} + \dotsb + w_j)
\]
corresponds to the functor which sends a split filtration $M \in \Filt^n_{k\Sc_0}(\Sc)$ to the $\Sc$-module $M_j/M_i$. 

\begin{cor}\label{cor:stratum of subquotients}
Let $w_1,\dots,w_n: \Sc_0 \to \N$ be dimension vectors and let $M, N \in \Tf_0(w_1,\dots,w_n)$. If $M$ and $N$ are equivalent, then the subquotients $\subquot^n_{i,j}(M)$ and $\subquot^n_{i,j}(N)$ belong to the same stratum of $\Mf_0(w_{i+1} + \dotsb + w_j)$ for all $0 \leq i < j \leq n$.
\end{cor}

\begin{proof}
    This follows from \cite[Theorem 2.7]{KellerScherotzke16} and the fact that the composition
    \[
    \Phi^1 \circ \tot^{j-i} \circ \sub^{n-i,j-i} \circ \quot^{n,n-i}
    \]
    factors through $\Phi^n$ by Theorems \ref{thm:compatibility without tot} and \ref{thm:compatibility with tot}.
\end{proof}

The converse does not hold. If it were true, then $\Tf_0(w_1,\dots,w_n)$ would contain only finitely many equivalence classes, since Nakajima's stratification always consists of finitely many strata. However, because the algebra $B_n$ is rarely of finite representation type when $n \geq 2$, it is not hard to find dimension vectors for which $\Tf_0(w_1,\dots,w_n)$ has infinitely many equivalence classes. We give an example below. 

\begin{example}\label{example:infinite partition}
Take $n = 2$ and let $Q$ be the following Dynkin quiver of type $\mathsf{D}_4$:
\[\begin{tikzcd}[column sep=small, row sep=0.1em]
	&& 3 \\
	Q: \ 1 & 2 \\
	&& 4
	\arrow[from=2-2, to=1-3]
	\arrow[from=2-2, to=2-1]
	\arrow[from=2-2, to=3-3]
\end{tikzcd}\]
The algebra $B_2$ is isomorphic to the path algebra of the following quiver modulo the relations making the three squares commutative:
\[\begin{tikzcd}[column sep=small, row sep=0.1em]
	&&& {(3,1)} \\
	{(1,1)} && {(2,1)} && {(4,1)} \\
	&&& {(3,2)} \\
	{(1,2)} && {(2,2)} && {(4,2)}
	\arrow[from=2-3, to=1-4]
	\arrow[from=2-3, to=2-1]
	\arrow[from=3-4, to=1-4]
	\arrow[from=4-1, to=2-1]
	\arrow[from=4-3, to=2-3]
	\arrow[from=4-3, to=3-4]
	\arrow[from=4-3, to=4-1]
	\arrow[from=4-3, to=4-5]
	\arrow[from=4-5, to=2-5]
    \arrow[from=2-3, to=2-5, crossing over]
\end{tikzcd}\]
The subquiver given by the vertices $(2,1)$ and $(i,2)$ for $1 \leq i \leq 4$ is an affine Dynkin quiver of type $\widetilde{\mathsf{D}}_4$. Its path algebra is a quotient of $B_2$, so $B_2$ is of infinite representation type. For example, it has infinitely many non-isomorphic representations of the form
\[\begin{tikzcd}[column sep=small, row sep=0.1em]
	&&& 0 \\
	0 && k && 0 \\
	&&& k \\
	k && {k^2} && k
	\arrow[from=1-4, to=2-3]
	\arrow[from=1-4, to=3-4]
	\arrow[from=2-1, to=2-3]
	\arrow[from=2-1, to=4-1]
	\arrow[from=2-3, to=4-3]
	\arrow[from=2-5, to=4-5]
	\arrow[from=3-4, to=4-3]
	\arrow[from=4-1, to=4-3]
	\arrow[from=4-5, to=4-3]
    \arrow[from=2-5, to=2-3, crossing over]
\end{tikzcd}\]
We prove that any such representation is isomorphic to $\Phi^2(X)$ in $\Dc_{B_2}$ for some finite-dimensional $\Sc\filt{2}$-module $X$. Moreover, the dimension vectors $w_1, w_2: \Sc_0 \to \N$ of $X$ can be chosen to be the same for all these representations. As a result, we obtain that the equivalence relation on $\Tf_0(w_1,w_2)$ induced by $\Phi^2$ has infinitely many equivalence classes.

Since the representations above are iterated extensions of simple modules, Corollary \ref{cor:syzygies give all Gorenstein modules} and Remark \ref{rem:extension of syzygies is syzygy} reduce our claim to showing that every simple $B_2$-module is in the essential image of $\Phi^2$. Let $x_i \in (\Z Q)_0$ be such that $H(x_i)$ is the simple $kQ$-module corresponding to the vertex $i \in Q_0$. If $T_{i,j}$ denotes the simple $B_2$-module corresponding to the vertex $(i,j)$ of the quiver above, then $T_{i,2} = \mu_2(H(x_i))$ and we have a distinguished triangle
\[
\mu_2(H(x_i)) \longrightarrow \mu_1(H(x_i)) \longrightarrow T_{i,1} \longrightarrow \Sigma\mu_2(H(x_i))
\]
in $\Dc_{B_2}$. Applying the equivalence $G_n$ to the triangle above, we deduce that $G_n(T_{i,1})$ is an extension of $\Omega S_{\sigma_2(\Sigma x_i)}$ by $\Omega S_{\sigma_1(x_i)}$ in $\gprs(\Sc\filt{2})$. By Corollary \ref{cor:syzygies give all Gorenstein modules}, there is an extension $M_i$ of $S_{\sigma_2(\Sigma x_i)}$ by $S_{\sigma_1(x_i)}$ such that $\Omega M_i \cong G_n(T_{i,1})$. We conclude that $\Phi^2(M_i) \cong T_{i,1}$ and $\Phi^2(S_{\sigma_2(x_i)}) \cong T_{i,2}$, as desired. This proves our claim and shows that the dimension vectors $w_1$ and $w_2$ from the previous paragraph can be taken to be
\[
w_1(u) = \begin{cases}
    1, &\textrm{if } u = \sigma(x_2),\\
    0, &\textrm{otherwise},\\
\end{cases} \quad \textrm{and} \quad w_2(u) = \begin{cases}
    2, &\textrm{if } u = \sigma(x_2),\\
    1, &\textrm{if } u = \sigma(x_1),\sigma(x_3),\sigma(x_4) \textrm{ or } \sigma(\Sigma x_2),\\
    0, &\textrm{otherwise},\\
\end{cases}
\]
for $u \in \Sc_0$.
\end{example}

\sloppy
\printbibliography

\end{document}